\documentclass[11pt,reqno]{amsart}
\usepackage{amsfonts,amstext}
\usepackage{amsmath,amstext,amssymb,amscd,mathrsfs,amsthm}
\usepackage{enumerate}
\usepackage[dvipsnames,usenames]{xcolor}

\newtheorem{thm}{Theorem}[section]
\newtheorem{lem}[thm]{Lemma}

\newtheorem{assumption}[thm]{Assumption}
\theoremstyle{definition}
\newtheorem{defn}[thm]{Definition}

\theoremstyle{remark}
\newtheorem{rmk}[thm]{Remark}
\numberwithin{equation}{section}

\usepackage[allbordercolors={1 1 1}]{hyperref}

\usepackage[disable]{todonotes}

\def\grad{\nabla}

\renewcommand{\H}[1][1]{W^{#1,2}(\Omega)}

\renewcommand{\to}{\mathrel{\rightarrow}}
\newcommand{\tinto}{\mathrel{\overset{\gamma}{\to}}}

\newcommand{\R}{\mathbb{R}}

\def\g{\gamma}

\def\t{\tau}

\def\G{\Gamma}
\def\O{\Omega}

\def\e{\epsilon}

\newcommand{\norm}[1]{\left\|#1\right\|}
\newcommand{\abs}[1]{\left|#1\right|}

\def\<{\langle}
\def\>{\rangle}

\def\H{H^1_{\G_0}}

\begin{document}

\author[B. Feng]{Baowei Feng}
\address{Department of Mathematics,
Southwestern University of Finance and Economics,
Chengdu 611130, Sichuan, P. R. China }
 \email{bwfeng@swufe.edu.cn}

 \author[Y.  Guo]{Yanqiu Guo}
\address{Department of Mathematics and Statistics,
Florida International University, Miami FL 33199, USA }
 \email{yanguo@fiu.edu}

\author[M. A. Rammaha]{Mohammad A. Rammaha}
\address{Department of Mathematics, University of Nebraska--Lincoln, Lincoln, NE  68588-0130, USA} \email{mrammaha1@unl.edu}

\title[a structure acoustics model]{On the asymptotic behavior of solutions to a structure acoustics model }

\date{February 12, 2022}

\maketitle

\begin{abstract} This article concerns the long term behavior of solutions to a structural acoustic model consisting of a semilinear wave equation defined on a smooth bounded domain $\Omega\subset\mathbb{R}^3$ which is coupled with a Berger plate equation acting on a flat portion of the boundary of $\Omega$. The system is influenced by several competing forces, in particular a source term acting on the wave equation which is allowed to have a supercritical exponent. 

Our results build upon those obtained by Becklin and Rammaha \cite{Becklin-Rammaha2}.  With some restrictions on the parameters in the system and with careful analysis involving the Nehari manifold we obtain global existence of potential well solutions and establish either exponential or algebraic decay rates of energy, dependent upon the behavior of the damping terms.  The main novelty in this work lies in our stabilization estimate, which notably does not generate lower-order terms. Consequently, the proof of the main result is shorter and more concise. 
\end{abstract}

\maketitle

\section{Introduction}\label{S1}

\subsection{The Model}
Structural acoustic interaction models have rich history. These models are well known in both the physical and mathematical literature and go back to the canonical models considered in \cite{Beale76,Howe1998}. For instance, the model studied by Beale  \cite{Beale76} is constructed as follows. Suppose $\O \subset \R^3$ is a bounded smooth domain filled with a fluid  which is at rest, except for an acoustic wave motion. If $u(x,t)$ is the velocity potential of the fluid, so that 
$-\grad u(x,t)$ is the particle velocity, then $u$ satisfies the wave equation
\[
u_{tt} = c^2\Delta u \,\, \text{ in } \O \times (0,T),
\]
where $c$ is the speed of sound in the medium. Further assume that $\G=\partial\O$ is not rigid but subject to small oscillations, where each point $x\in \partial\O$ reacts to the pressure wave like a damped oscillator. Then, the normal displacement $w$ of the boundary into the domain satisfies the ODE:
\[
m(x) w_{tt}(x,t)+ d(x) w_t(x,t) +k(x) w(x,t)=-\rho u_t(x,t)|_{\G} \,\, \text{ on } \G \times (0,T),
\]
where $\rho$ is the density of the fluid. In addition, by assuming the boundary $\G$ is impenetrable, then from the continuity of velocity at the boundary one has:
\[
\partial_\nu u=w_t \,\, \text{ on }\G\times(0,T).
\]

Motivated by the above model,  we study a structure acoustic model influenced with nonlinear forces. Precisely, we study the coupled system of PDEs:
\begin{align}\label{PDE}
\begin{cases}
u_{tt}-\Delta u +g_1(u_t)=f(u) &\text{ in } \O \times (0,T),\\[1mm]
w_{tt}+\Delta^2w+g_2(w_t)+u_t|_{\G}=h(w)&\text{ in }\G\times(0,T),\\[1mm]
u=0&\text{ on }\G_0\times(0,T),\\[1mm]
\partial_\nu u=w_t&\text{ on }\G\times(0,T),\\[1mm]
w=\partial_{\nu_\G}w=0&\text{ on }\partial\G\times(0,T),\\[1mm]
(u(0),u_t(0))=(u_0,u_1),\hspace{5mm}(w(0),w_t(0))=(w_0,w_1),
\end{cases}
\end{align}
where the  initial data reside in the finite energy space, i.e.,
$$(u_0, u_1)\in H^1_{\G_0}(\O) \times L^2(\O) \, \text{ and }(w_0, w_1)\in H^2_0(\G)\times L^2(\G),$$
where the space $H^1_{\G_0}(\O)$ is defined in (\ref{H1}) below.

Here,  $\O\subset\R^3$ is a bounded, open, connected domain with a smooth boundary
$\partial\O=\overline{\G_0\cup\G}$, where $\G_0$ and $\G$ are two disjoint, open, connected sets of positive Lebesgue  measure.  Moreover, $\G$ is a \emph{flat} portion of the boundary of $\O$ and is referred to as the elastic wall. The part $\G_0$ of the boundary $\partial\O$ describes a rigid  wall, while  the coupling takes place on the flexible wall  $\G$.

It is interesting that problem \eqref{PDE} includes geometric elements of one, two and three dimensions. 
In particular, $\Omega$ is a three-dimensional region in which a nonlinear wave equation for $u$ is defined. 
The boundary of $\Omega$ contains a flat portion $\Gamma$, a two-dimensional plane, in which a plate equation for $w$ is defined. Finally, the boundary of $\Gamma$ is a smooth curve $\partial \Gamma$, which is a one-dimensional geometric object, on which a boundary condition for $w$ is imposed.

The nonlinearities $f$ and $h$ represent source terms acting on the wave and plate equations respectively, where $f(u)$ is of a supercritical order and both source terms are allowed to have  ``bad" signs  which  may cause instability (blow up) in  finite time.  In addition, the system is influenced by two other competing forces, namely $g_1(u_t)$ and $g_2(w_t)$ representing frictional damping terms acting on the wave and plate equations, respectively.  The presence of frictional damping is necessary to stabilize the system- otherwise nonlinear sources can lead to blow up in finite time.  The vectors $\nu$ and $\nu_\G$ denote the outer normals to $\G$ and $\partial\G$; respectively.

Models such as (\ref{PDE}) arise in the context of modeling gas pressure in an
acoustic chamber which is surrounded by a combination of rigid and
flexible walls. The pressure in the chamber is described by the solution
to a wave equation, while vibrations of the flexible wall are described
by the solution to a coupled a Berger plate equation. We refer the reader to \cite{Chueshov-1999} and the references quoted therein on the Berger model.

In system (\ref{PDE}), the coupling of the wave equation for $u$ and the plate equation for $w$ are through the term $u_t|_{\Gamma}$. 
Note, the solution $(u,u_t)$ for the wave equation belongs to the finite energy space $H^1_{\G_0}(\O) \times L^2(\O)$. 
Therefore, $u_t$ belongs to the space $L^2(\Omega)$, and in general, without additional regularity, one cannot take the trace of an $L^2(\Omega)$ function on $\Gamma$. Therefore, in the definition of weak solutions below, namely, Definition \ref{def:weaksln}, we test $u_t|_{\Gamma}$ with a test function $\psi$, and pass the time derivative to $\psi$ in the spirit of weak derivative, and obtain the term $\int_{\Gamma} u|_{\Gamma}(t) \psi_t(t)d\Gamma -
\int_{\Gamma} u|_{\Gamma}(0) \psi_t(0)d\Gamma-\int_0^t \int_{\Gamma} u|_{\Gamma} \psi_t d\Gamma d\tau$,
which makes perfect sense since $u\in H^1_{\G_0}(\O) $ regular enough to take the boundary trace on $\Gamma$. 
Moreover, in the paper  \cite{Becklin-Rammaha1}, Becklin and Rammaha studied a similar model with restoring source terms but no damping term, 
and they pointed out that there is a hidden regularity for $u_t$ on $\Gamma$, which was $u_t|_{\Gamma \times (0,T)} \in H^{-2/3}(\Gamma \times (0,T))$ by assuming additional regularity on initial data and that the source term was subcritical.
The point is that the term $u_t|_{\Gamma}$ is an important element connecting the dynamics of $u$ and $w$ and can cause difficulties in the analysis of model (\ref{PDE}).

One of the novelties  of this manuscript lies in that the important stabilization estimate (\ref{4-12}) presented in Lemma \ref{lem4-1} does not contain lower-order terms on the right-hand side of the inequality. This feature greatly shortens the proof, because the standard lengthy compactness-uniqueness argument to absorb the lower-order terms is completely avoided. In particular, the estimate in subsection \ref{est-I2} includes some new idea and is applicable to the study of energy decay of other related systems.

\vspace{0.1 in}

\subsection{Literature Overview}
 As we mentioned earlier, structural acoustic interaction models have rich history and they go back to the canonical models considered in \cite{Beale76,Howe1998}. 
In the context of stabilization and controllability of structural acoustic models there is a very large body of literature. We refer the reader to the monograph by Lasiecka \cite{Las2002} which provides a  comprehensive overview and quotes many works on these topics. Other related contributions worthy of mention include \cite{Avalos2,Avalos1,Avalos3,Avalos4,Cagnol1,MG1,MG3,LAS1999}.

This manuscript focuses on potential well solutions of system \eqref{PDE}. The study of potential well solutions for nonlinear hyperbolic equations has a long history. For example, Payne and Sattinger \cite{PS} considered a nonlinear hyperbolic equation in the canonical form:
\begin{align}  \label{canonical}
u_{tt} = \Delta u + f(u),  \;\; \text{with} \;\; u(0)=u_0,  \;\;  u_t(0)=u_1,
\end{align}
such that $u=0$ on the boundary of a smooth bounded domain $\Omega\subset \mathbb R^n$. An important quantity for equation (\ref{canonical}) is the potential energy functional 
$J(u)=\frac{1}{2}\|u\|_{L^2(\Omega)}^2 -  \int_{\Omega} F(u) dx$ where $F(u)= \int_0^u f(s) ds$. 
The depth of the potential well $d$ can be defined as the mountain pass level, i.e., $d=\inf_{u \not =0} \sup_{\lambda>0} J(\lambda u)$.
Assume the initial total energy and the initial potential energy $J(u_0)$ are both less than $d$.
Then, if $\frac{d}{d\lambda} J(\lambda u) \geq 0$ for $0<\lambda \leq 1$, i.e., $J$ is increasing along any ray from the origin, then the weak solution is global in time; but, if $\|u\|_{L^2}^2 <  \int_{\Omega} uf(u) dx$, the solution blows up in finite time (see \cite{PS}). 
The global existence of potential well solutions for \eqref{canonical} can also be found in an earlier work  by Sattinger \cite{Sattiinger68}. 
In the same spirit, we show the global existence of potential well solutions for the structure acoustics model (\ref{PDE}), and the blow-up of solutions will be considered in a future paper.

System (\ref{PDE}) involves competing forces. In particular, nonlinear source terms $f(u)$ and $h(w)$ are competing with nonlinear frictional-type damping terms $g_1(u_t)$ and $g_2(w_t)$. 
A classical result on wave equations with nonlinear damping term $|u_t|^{m-1}u_t$ and source term $|u|^{p-1}u$ was established by Georgiev and Todorova \cite{GT} in 1994.
In particular, if the damping term dominates the source term ($m\geq p$), then weak solutions are global in time; whereas, if the source term surpasses the damping term $p>m$, then weak solutions may blow up if initial energy is large enough. 
An extension to a supercritical source term ($3\leq p < 6$ in 3D) was obtained by Bociu and Lasiecka \cite{BL3, BL2, BL1}.
Furthermore, for a nonlinear wave equation with a source term of rapid polynomial growth rate (including the range $p\geq 6$ in a 3D periodic domain), Guo \cite{Guo98} showed the global well-posedness of weak solutions, provided the damping term has sufficiently fast growth rate 
($m\geq \frac{3}{2}p - \frac{5}{2}$ if $p\geq 6$). There are many other interesting works concerned with the competition of source terms and various types of damping terms in nonlinear wave equations. See, for example, \cite{BLR1, GR, GRSTT,MOHNICK1,MOHNICK2,PRT-1,PRT-p-Laplacain} and references therein.

\section{Preliminaries and main results}

\subsection{Notation}

\noindent{}Throughout the paper the following notational conventions for $L^p$ space norms and standard inner products will be used:
\begin{align*}
&||u||_p=||u||_{L^p(\O)}, &&(u,v)_\O = (u,v)_{L^2(\O)},\\
&|u|_p=||u||_{L^p(\G)},&&(u,v)_\G = (u,v)_{L^2(\G)}.
\end{align*}
We also use the notation  $\g u$ to denote the \emph{trace} of $u$ on $\G$ and we write $\frac{d}{dt}(\g u(t))$ as $\g u_t$ or $\g u'$. Occasionally, we also use the notation $u_{|_\G}$ to mean $\g u$.
As is customary, $C$ shall always denote a positive constant which may change from line to line.

Further, we put
\begin{align}  \label{H1}
\H(\O):=\{u\in H^1(\O):u|_{\G_0}=0\}.
\end{align}
It is well-known that the standard norm
$\norm{u}_{\H(\O)}$ is equivalent to $\norm{\grad u}_2$. Thus, we put:
\begin{align} \label{normH1}
\norm{u}_{\H(\O)} = \norm{\grad u}_2,\,\, (u,v)_{\H(\O)}=(\grad u,\grad v)_\O.
\end{align}
For a similar reason, we put:
\begin{align}    \label{normH2}
\norm{w}_{H_0^2(\G)}= \abs{\Delta w}_2, \,\, (w,z)_{H^2_0(\G)}=(\Delta w,\Delta z)_\G.
\end{align}
For convenience and brevity, we shall  frequently use the notation:
\[
\norm{u}_{1, \O}= \norm{\grad u}_2,\,\,\, \norm{w}_{2, \G}= \abs{\Delta w}_2.
\]
\noindent{}With $Y$ is a Banach space, we denote the duality pairing between the dual space $Y'$ and $Y$ by
$\<\psi,y\>_{Y',Y}$, or simply by  $\<\cdot,\cdot\>$. That is,
\begin{align*}
\< \psi,y \> =\psi(y)\text{ for }y\in Y,\, \psi \in Y'.
\end{align*}

\noindent{}Throughout the paper, the following Sobolev imbeddings will be used often without mention:
\begin{align*}
\begin{cases}
H^{1-\epsilon}(\O)\hookrightarrow L^{\frac{6}{1+2\epsilon}}(\O)\text{ for }\epsilon\in[0,1],  \vspace{.05in}\\
H^{1-\epsilon}(\O) \tinto H^{\frac{1}{2}-\epsilon}(\G)\hookrightarrow L^{\frac{4}{1+2\epsilon}}(\G)\text{ for }
\epsilon\in [0,\frac{1}{2}],  \vspace{.05in} \\
H^1(\Gamma)\hookrightarrow L^q(\G) \text{ for all } 1\leq q<\infty.
\end{cases}
\end{align*}
Finally, we remind the reader with the following interpolation inequality:
\begin{equation} \label{inter}
\norm{u}^2_{H^{\theta}(\O)} 
\leq \e \norm{u}^2_{1,\O}+C(\e,\theta)\norm{u}_2^2,
\end{equation}
for all $0\leq \theta <1$ and $\e>0$.

\vspace{0.1 in}

\subsection{Weak Solutions} Throughout the paper, we study  (\ref{PDE}) under the following assumptions.

\begin{assumption}\label{ass}
\hfill
\begin{description}
\item[Damping terms] $g_1, \, g_2:\R \rightarrow \R$ are continuous and monotone increasing functions with $g_1(0)=g_2(0)=0$.  In addition, the following growth conditions at infinity hold: there exist positive constants $\alpha$ and $\beta$ such that, for $|s|\geq1$,
\begin{align*}
&\alpha|s|^{m+1}\leq g_1(s)s\leq\beta|s|^{m+1},\text{ with }m\geq1,\\
&\alpha|s|^{r+1}\leq g_2(s)s\leq\beta|s|^{r+1},\text{ with }r\geq1.
\end{align*}
\item[Source terms] $f$ and $h$ are functions in $C^1(\R)$ such that
\begin{align*}&|f'(s)|\leq C(|s|^{p-1}+1),\text{ with }1\leq p< 6,\\
&|h'(s)|\leq C(|s|^{q-1}+1),\text{ with }1\leq q<\infty.
\end{align*}
\item[Parameters] \,\, $p\frac{m+1}{m}<6$.
\end{description}
\end{assumption}

\begin{rmk}\label{rem1}
 As the following bounds will be used often throughout the paper, it is worthy of note that the above assumption  implies that
    \begin{align}\label{1.2}
    \begin{cases}
     |f(u)|\leq C(|u|^p+1),\quad |f(u)-f(v)|\leq C(|u|^{p-1}+|v|^{p-1}+1)|u-v|, \vspace{.1in} \\
    |h(w)| \leq C(|w|^q +1),\,\, |h(w)-h(z)|\leq C(|w|^{q-1}+|z|^{q-1}+1)|w-z|.
    \end{cases}
    \end{align}
\end{rmk}

The following assumption will be needed for establishing the result on the uniqueness of solutions.
\begin{assumption}\label{ass2}
For $p>3$, we assume that  $f\in C^2(\R)$ with $|f''(u)|\leq C(|u|^{p-2}+1)$ for all $u\in\R$.
\end{assumption}

\noindent{} We begin by introducing the definition of a suitable weak solution for \eqref{PDE}.
\begin{defn}\label{def:weaksln}
A pair of functions $(u,w)$ is said to be a \emph{weak solution} of \eqref{PDE} on the interval $[0,T]$ provided:
    \begin{enumerate}[(i)]
        \setlength{\itemsep}{5pt}
        \item\label{def-u} $u\in C([0,T];H^1_{\G_0}(\O))$, $u_t\in C([0,T];L^2(\O))\cap L^{m+1}(\O\times(0,T))$,
        \item\label{def-v} $w\in C([0,T];H^2_0(\G))$, $w_t\in C([0,T];L^2(\G))\cap L^{r+1}(\G\times(0,T))$,
        \item\label{def-uic} $(u(0),u_t(0))=(u_0,u_1) \in H^1_{\G_0}(\O)\times L^2(\O)$,
        \item\label{def-wic} $(w(0),w_t(0))=(w_0,w_1) \in H^2_0(\G)\times L^2(\G)$,
        \item\label{def-ws} The functions $u$ and $w$ satisfy the following variational  identities
        for all $t\in[0,T]$:
\begin{align}\label{wkslnwave}
(u_{t}(t),  \phi(t))_\O & - (u_1,\phi(0))_\O-\int_0^t ( u_t(\tau), \phi_t(\tau) )_\O d\tau
+\int_0^t (\nabla u(\tau), \nabla\phi(\tau) )_\O d\tau \notag \\
&-\int_0^t  (w_t(\tau), \g\phi(\tau) )_\G d\tau+\int_0^t\int_\Omega g_1(u_t(\tau))\phi(\tau) dxd\tau  \notag\\
&=\int_0^t\int_\Omega f(u(\tau))\phi(\tau) dxd\tau,
\end{align}

\begin{align}\label{wkslnplt}
(w_t(t) & + \g u(t),\psi(t) )_\G  -(w_1 +\g u(0) ,\psi(0))_\G -\int_0^t (w_t(\tau), \psi_t(\tau) )_\G d\tau \notag \\
& -\int_0^t (\g u(\tau), \psi_t(\tau) )_\G d\tau+\int_0^t (\Delta w(\tau), \Delta\psi(\tau) )_\G d\tau \notag \\
&+ \int_0^t\int_{\G}g_2(w_t(\tau))\psi(\tau) d\G d\tau=\int_0^t\int_{\G}h(w(\tau))\psi(\tau) d\G d\tau,
\end{align}
for all test functions $\phi$ and $\psi$ satisfying:
$\phi\in C([0,T];H^1_{\G_0}(\O))  \cap L^{m+1}(\O\times(0,T))$, $\psi\in C\left([0,T];H^2_0(\G)\right)$  with
$\phi_t\in L^1(0,T;L^2(\O))$, and $\psi_t\in L^{1}(0,T;L^2(\G))$.
    \end{enumerate}
\end{defn}

As mentioned earlier, our work in this paper is based on the existence results which were established in \cite{Becklin-Rammaha2}. For the reader's convenience, we first summarize the important results in \cite{Becklin-Rammaha2}.

\begin{thm} [{\bf Local and  global weak solutions \cite{Becklin-Rammaha2}}]  \label{t:1}
    Under the validity of Assumption \ref{ass}, then there exists a local weak soluition $(u,w)$ to \eqref{PDE} in the sense of Definition \ref{def:weaksln}, defined on $[0,T_0]$ for some $T_0>0$ depending on the initial quadratic energy $E(0)$, where the quadratic energy is defined as
\begin{align}     \label{qua}
E(t) :=\frac{1}{2}\left(\|u_t(t)\|_2^2 +|w_t(t)|_2^2    +\|\nabla u(t)\|_2^2+|\Delta w(t)|_2^2\right).
\end{align}
\begin{itemize}
\item ~~$(u,w)$ satisfies the following energy identity for all $t\in [0,T_0]$:
\begin{align}\label{energy}
    E(t)&+\int_0^t\int_\O g_1(u_t)u_tdxd\t+\int_0^t\int_\G g_2(w_t)w_td\G d\t\notag\\
    &=E(0)+\int_0^t \int_\O f(u)u_tdxd\t+\int_0^t\int_\G h(w)w_td\G d\t.
\end{align}
\item In addition to Assumption  \ref{ass}, if we assume that $u_0\in L^{p+1}(\O)$, $p\leq m$ and $q\leq r$, then the said solution $(u,w)$  is a global weak solution and $T_0$ can be taken arbitrarily large.
\item If Assumption \ref{ass} is valid, and if we additionally assume that $u_0\in L^{3(p-1)}(\O)$ and $m \geq 3p-4$ when  $p>3$, then weak solutions of \eqref{PDE} are unique.

\item If Assumptions \ref{ass} and \ref{ass2} are valid, and if we further assume that $u_0\in L^{\frac{3(p-1)}{2}}(\O)$ then, weak solutions of \eqref{PDE} are unique.

\end{itemize}
\end{thm}

\vspace{0.1 in}

\subsection{Potential Well}
In this section we introduce the concepts of Nehari manifold and potential well. 

Let us first impose additional assumptions on source terms $f$ and  $h$.
\begin{assumption}\label{ass1}
\hfill
\begin{itemize}
\item
 There exists a nonnegative function
$F(s)\in C^1(\mathbb{R})$ such that $F'(s)=f(s)$, and $F$ is
homogeneous of order $p+1$, i.e., $F(\lambda s)=\lambda^{p+1}F(s)$,
for $\lambda>0,\ u\in\mathbb{R}$.
\item There exists a nonnegative function $H(s)\in C^1(\mathbb{R})$
such that $H'(s)=h(s)$, and $H$ is homogeneous of order $q+1$, i.e.,
$H(\lambda s)=\lambda^{q+1}H(s)$, for $\lambda>0,\ s\in\mathbb{R}$.
\end{itemize}
\end{assumption}

\begin{rmk}\label{rmk3-1}
From Euler homogeneous function theorem we infer that
\begin{align}\label{3-3}
uf(u)=(p+1)F(u),\ \ wh(w)=(q+1)H(w).
\end{align}
It follows from Assumption \ref{ass} that there exists a positive
constant $M$ such that
\begin{align}\label{3-4}
F(u)\leq M(|u|^{p+1}+1), H(w)\leq M(|w|^{q+1}+1),
\end{align}
which, noting that $F$ and $H$ are homogeneous, yields that
\begin{align}\label{3-5}
F(u)\leq M|u|^{p+1},\ \ H(w)\leq M|w|^{q+1}.
\end{align}
Moreover, we know that $f$ is homogeneous of order $p$ and
$h$ is homogeneous of order $q$, and we see from (\ref{3-3}) and (\ref{3-5})
that
\begin{align}\label{3-6}
|f(u)|\leq (p+1) M|u|^{p},\ \, |h(w)|\leq (q+1)M|w|^q.
\end{align}
\end{rmk}

Put $X=H^1_{\Gamma_0}(\Omega)\times H^2_0(\Gamma)$. 
According to (\ref{normH1}) and (\ref{normH2}), $X$ is endowed by the natural norm:  
\begin{align}    \label{normX}
\norm{(u,w)}_X  =   (\norm{\nabla u}_2^2 +  |\Delta w|_2^2)^{1/2}. 
\end{align}
Define the nonlinear functional $\mathcal{J}:X\to \mathbb{R}$ by
\begin{align}\label{3-8}
\mathcal{J}(u,w):=\frac{1}{2}(\|\nabla u\|^2_2 + |\Delta w|^2_2) -\int_\Omega F(u)dx-\int_\Gamma H(w) d\Gamma.
\end{align}
Then,  the \emph{potential energy} of the system is given by $\mathcal J(u(t),w(t))$. 
The Fr\'{e}chet derivative of $\mathcal{J}$ at $(u,w)\in X$ is given
by
\begin{align}\label{fre}
\langle\mathcal{J}'(u,w),(\phi,\psi)\rangle
=\int_\Omega\nabla
u\cdot\nabla\phi dx+\int_\Gamma\Delta w\cdot\Delta\psi d\Gamma -\int_\Omega f(u)\phi dx-\int_\Gamma h(w)\psi d\Gamma,
\end{align}
for $(\phi,\psi)\in X$.

The \emph{Nehari manifold} associated with the functional $\mathcal J$ is defined as
\begin{align}   \label{Nehari}
\mathcal{N}:=\left\{(u,w)\in X\backslash\{(0,0)\}:
\langle\mathcal{J}'(u,w),(u,w)\rangle=0\right\}.
\end{align}
The Nehari manifold is a manifold of functions, whose definition is motivated by the work of Zeev Nehari \cite{nehari1960class, nehari1961characteristic} 
for nonlinear second-order differential equations.

Due to (\ref{Nehari}), \eqref{fre} and \eqref{3-3}, we have
\begin{align}\label{3-10}
\mathcal{N}=\Big\{(u,w)\in X\backslash\{(0,0)\}: 
\|\nabla u\|^2_2+|\Delta w|^2_2 & =(p+1)\int_\Omega F(u)dx \nonumber \\ & +(q+1)\int_\Gamma
H(w)d\Gamma \Big\}.
\end{align}

We define the \emph{potential well} associated with the potential energy $\mathcal{J}(u,w)$ by 
\begin{align}    \label{def-W}
\mathcal{W}&:=\{(u,w)\in X: \mathcal{J}(u,w)<d\},
\end{align}
where the \emph{depth of the potential well} $\mathcal{W}$ is defined as
\begin{align}    \label{depth}
d:=\inf_{(u,w)\in\mathcal{N}}\mathcal{J}(u,w).
\end{align}

In order to make sure that the set $\mathcal W$ is non-empty, we need to verify that $d$ is strictly positive. 
The positivity of $d$ is provided by the following lemma, under certain assumptions.
\begin{lem}\label{lem3-1}
Let Assumption \ref{ass} and Assumption \ref{ass1} hold. Let
$1<p\leq 5$ and $q>1$, then $d>0$.
\end{lem}
\begin{proof}
Fix $(u,w) \in \mathcal{N}$. In view of \eqref{3-8} and \eqref{3-10}, we get
\begin{align}\label{3-11-1}
\mathcal{J}(u,w)\geq\left(\frac{1}{2}-\frac{1}{c}\right)(\|\nabla
u\|_2^2+|\Delta w|^2_2),
\end{align}
where $c:=\min\{p+1,q+1\}>2$. 
Since $p\leq 5$, it follows from \eqref{3-5}, \eqref{3-10} and embedding inequalities that
\begin{align}
\|\nabla u\|^2_2+|\Delta w|^2_2\leq
C_{p,q}(\|u\|^{p+1}_{p+1}+|w|^{q+1}_{q+1})  \leq C(\|\nabla
u\|^{p+1}_2+|\Delta w|^{q+1}_2),\nonumber
\end{align}
which gives us
\begin{align}\label{3-11-2}
\|(u,w)\|^2_X\leq C(\|(u,w)\|^{p+1}_X+\|(u,w)\|^{q+1}_X).
\end{align}
Noting $(u,w)\neq(0,0)$, we infer from \eqref{3-11-2} that
$$
\|(u,w)\|^{p-1}_X+\|(u,w)\|^{q-1}_X\geq \frac{1}{C}.
$$
Then, $\|(u,w)\|_X\geq s_0>0$, where $s_0$ is the unique positive
root of the equation $s^{p-1}+s^{q-1}=\frac{1}{C}$. It follows from
\eqref{3-11-1} that
$$
\mathcal{J}(u,w)\geq \left(\frac{1}{2}-\frac{1}{c}\right)s_0^2,\
\;\;\text{for all}\;\; (u,w)\in\mathcal{N},
$$
completing the proof.
\end{proof}

We remark that the potential well $\mathcal W$ and the Nehari manifold $\mathcal N$ are disjoint sets, because of (\ref{def-W}) and (\ref{depth}). 
That is, 
\begin{align}   \label{disjointWN}
\mathcal W \cap \mathcal N =   \emptyset.
\end{align}
In other words, if $(u,w) \in X \backslash \{(0,0)\}$ such that $$\|\nabla u\|^2_2+|\Delta w|^2_2=(p+1)\int_\Omega F(u)dx+(q+1)\int_\Gamma H(w)d\Gamma,$$
then $(u,w) \not \in \mathcal W$.

The potential well $\mathcal W$ can be decomposed into two parts: the ``stable" part $\mathcal W_1$ and the ``unstable" part $\mathcal W_2$:
\begin{align*}
\mathcal{W}_1&=\left\{(u,w)\in\mathcal{W}:\|\nabla u\|^2_2+|\Delta w|^2_2>(p+1)\int_\Omega F(u)dx+(q+1)\int_\Gamma H(w)d\Gamma\right\} \cup\{(0,0)\},\\
\mathcal{W}_2&=\left\{(u,w)\in\mathcal{W}:\|\nabla u\|^2_2+|\Delta
w|^2_2<(p+1)\int_\Omega F(u)dx+(q+1)\int_\Gamma H(w)d\Gamma\right\}.
\end{align*}
Clearly, $\mathcal{W}_1\cap\mathcal{W}_2=\emptyset$ and . 
Moreover, because of \eqref{disjointWN}, we see that 
$$\mathcal{W}_1\cup\mathcal{W}_2=\mathcal{W}.$$
In this paper, we show that if the initial data $(u_0,w_0)\in \mathcal W_1$ and the initial total energy $\mathcal E(0)<d$, then the weak solution of system \eqref{PDE} is global in time. Also, we prove the uniform decay rate of energy under additional assumptions. 
In a future paper, we will study the finite-time blow-up when the initial data resides in $\mathcal W_2$.

Using a similar argument as the proof of Lemma 2.7 in \cite{GR2}, the depth of the potential well $d$ coincides with the mountain pass level. 
In particular, 
\begin{align}\label{3-11}
d:=\inf_{(u,w)\in\mathcal{N}}\mathcal{J}(u,w)=\inf_{(u,w)\in
X\backslash\{(0,0)\}}\sup_{\lambda\geq0}\mathcal{J}(\lambda (u,w)),
\end{align}
if the assumptions of Lemma \ref{lem3-1} are valid. 
The minimax method and mountain pass theorem in the theory of calculus of variations can be found in the book by Rabinowitz \cite{Minimax}.

\vspace{0.1 in}

\subsection{A closed subset of $\mathcal{W}_1$}   \label{subsetW}

Here, we construct a closed subset of $\mathcal{W}_1$. 
If initial data belong to such a subset, we are able to estimate energy decay rates in Section \ref{sec-decay}.

Throughout, we assume $1<p\leq 5$ and $q>1$.  Also, Assumption \ref{ass} and Assumption \ref{ass1} hold.

Thanks to the Sobolev embedding $H^1_{\Gamma_0}(\Omega)  \hookrightarrow L^{p+1}(\Omega)$ for $1<p\leq 5$ and 
$H^2_0(\Gamma) \hookrightarrow L^{q+1}(\Gamma)$ for $q>1$, we can define the best embedding constants:
\begin{align}\label{3-12}
K_1:=\sup_{u\in
H^1_{\Gamma_0}(\Omega)\backslash\{0\}}\frac{\|u\|^{p+1}_{p+1}}{\|\nabla
u\|^{p+1}_2},\ \ K_2:=\sup_{w\in
H^2_0(\Gamma)\backslash\{0\}}\frac{|w|^{q+1}_{q+1}}{|\Delta
w|^{q+1}_2}.
\end{align}

It follows  from \eqref{3-5} and \eqref{3-12} that
\begin{align}\label{3-13}
\mathcal{J}(u,w)&\geq\frac{1}{2}(\|\nabla u\|^2_2+|\Delta
w|^2_2)-M(\|u\|^{p+1}_{p+1}+|w|^{q+1}_{q+1})\nonumber\\
&\geq \frac{1}{2}(\|\nabla u\|^2_2+|\Delta w|^2_2)-MK_1\|\nabla
u\|^{p+1}_2-MK_2|\Delta w|^{q+1}_2\nonumber\\
&\geq \frac{1}{2}\norm{(u,w)}_X^2 - MK_1\|(u,w)\|^{p+1}_X-MK_2\|(u,w)\|^{q+1}_X,
\end{align}
where $X=H^1_{\Gamma_0}(\Omega)\times H^2_0(\Gamma)$
and $\norm{(u,w)}_X  =   (\norm{\nabla u}_2^2 +  |\Delta w|_2^2)^{1/2}$.  

Inequality \eqref{3-13} can be expressed as:
\begin{align}\label{3-13-1}
\mathcal{J}(u,w)\geq\Lambda(\|(u,w)\|_X),   \;\;  \text{for any} \,\, (u,v) \in X,
\end{align}
where  the function $\Lambda(s)$ is given by:
\begin{align}   \label{Lambda}
\Lambda(s):=\frac{1}{2}s^2-MK_1s^{p+1}-MK_2s^{q+1}.
\end{align}

In view of $p,q>1$, then
$$
\Lambda'(s)=s[1-MK_1(p+1)s^{p-1}-MK_2(q+1)s^{q-1}],
$$
has only one positive zero at $s^*$, where $s^*$ satisfies
\begin{align}\label{3-13-2}
MK_1(p+1)(s^*)^{p-1}+MK_2(q+1)(s^*)^{q-1}=1.
\end{align}
It is easy to see that $\Lambda(s)$ has a maximum value at
$s^*$ on $[0,\infty)$, i.e., $\sup_{s\in[0,\infty)}\Lambda(s)=\Lambda(s^*)>0$.

 Now, define
\begin{align}   \label{tW}
\tilde{\mathcal{W}}_1:=\{(u,w)\in X:
\|(u,w)\|_X<s^*,\mathcal{J}(u,w)<\Lambda(s^*)\}.
\end{align}

We remark that $\tilde{\mathcal{W}}_1$ is not the trivial set $\{(0,0)\}$, since for any $(u,w)\in X$, there  is a small real number $c>0$ 
such that the scaler multiple $c(u,w)$ belong to  $\tilde{\mathcal{W}}_1$.

\begin{lem}\label{lem3-2}
$\tilde{\mathcal{W}}_1$ is a subset of $\mathcal{W}_1$.
\end{lem}
\begin{proof}
For $(u,w)\in X\backslash\{(0,0)\}$, we infer  from \eqref{3-13-1} that
$\mathcal{J}(\lambda (u,w))\geq \Lambda(\lambda\|(u,w)\|_X)$ for all
$\lambda\geq0$. Then we have
$$
\sup_{\lambda\geq0}\mathcal{J}(\lambda (u,w))
 \geq   \sup_{\lambda\geq0}   \Lambda(\lambda\|(u,w)\|_X)  = \sup_{s\in[0,\infty)}\Lambda(s) = \Lambda(s^*).
$$
 By
\eqref{3-11}, we have 
\begin{align}    \label{Ld}
\Lambda(s^*)\leq d.
\end{align}

Combining \eqref{3-5} and \eqref{3-12}, and by using \eqref{3-13-2}, we
obtain that for any $(u,w) \in X \backslash \{(0,0)\}$ with $\|(u,w)\|_X< s^*$,
\begin{align}   \label{cW1}
&(p+1)\int_\Omega F(u)dx+(q+1)\int_\Gamma H(w)d\Gamma\nonumber\\
&\quad\leq (p+1)MK_1\|\nabla u\|^{p+1}_2+(q+1)MK_2|\Delta
w|^{q+1}_2\nonumber\\
&\quad\leq
\|(u,w)\|^2_X \Big[(p+1)MK_1\|(u,w)\|^{p-1}_X+(q+1)MK_2\|(u,w)\|^{q-1}_X\Big]\nonumber\\
&\quad<\|(u,w)\|^2_X \Big[(p+1)MK_1(s^*)^{p-1}+(q+1)MK_2(s^*)^{q-1}\Big]\nonumber\\
&\quad=\|(u,w)\|^2_X=\|\nabla u\|^2_2+|\Delta w|^2_2.
\end{align}

Because of \eqref{tW}, \eqref{Ld} and \eqref{cW1}, we conclude that $\tilde{\mathcal{W}}_1\subset\mathcal{W}_1$.
\end{proof}

For each sufficiently small $\delta>0$, we define a closed subset of $\tilde{\mathcal{W}}_1$ by
\begin{align}\label{3-14}
\tilde{\mathcal{W}}^{\delta}_1:=\{(u,w)\in X:
\|(u,w)\|_X\leq s^*-\delta, \, \mathcal{J}(u,w)\leq\Lambda(s^*-\delta)\}.
\end{align}
$\tilde{\mathcal{W}}^{\delta}_1$ is a closed set because the space $X$ is complete and $\mathcal J$ is continuous from $X$ to $\mathbb R$.
Clearly,
\begin{align*}
\tilde{\mathcal{W}}^{\delta}_1  \subset \tilde{\mathcal{W}}_1 \subset \mathcal W_1.
\end{align*}
In Section \ref{sec-decay}, we shall show the energy decay by assuming the initial data come from the closed set $\tilde{\mathcal{W}}^{\delta}_1$.
Such a closed subset of $\mathcal W_1$ was also used in \cite{GR2} by Guo and Rammaha to show the decay of energy for a system of coupled nonlinear wave equations. But, in paper \cite{GR2} the closed set $\tilde{\mathcal{W}}^{\delta}_1$ was used in the lengthy compactness-uniqueness argument to absorb the lower-oder terms; whereas, in this manuscript we adjust the stabilization estimate by taking advantage of the closed set $\tilde{\mathcal{W}}^{\delta}_1$ so that the appearance of lower-order terms is prevented, and so our proof is concise.

\vspace{0.1 in}

\subsection{Total energy}
The kinetic energy of system \eqref{PDE} is given by $\frac{1}{2}(\|u_t(t)\|^2_2+|w_t(t)|^2_2)$. 
Also, the potential energy is given by $\mathcal{J}(u(t),w(t))$, where the nonlinear function $\mathcal J$ is defined in \eqref{3-8}.
Naturally, we define the total energy $\mathcal E(t)$ as the summation of the kinetic energy and the potential energy, namely, 
\begin{align}\label{3-9}
\mathcal{E}(t): &=   \frac{1}{2}(\|u_t(t)\|^2_2+|w_t(t)|^2_2)+\mathcal{J}(u(t),w(t))  \notag\\
&=  \frac{1}{2}\left(\|u_t(t)\|_2^2+|w_t(t)|_2^2   +\|\nabla u(t)\|_2^2   +|\Delta w(t)|_2^2\right)  -\int_\Omega F(u)dx-\int_\Gamma H(w) d\Gamma     \notag\\
&= E(t)   -\int_\Omega F(u)dx-\int_\Gamma H(w) d\Gamma,
\end{align}
where the quadratic energy $E(t)$ is defined in (\ref{qua}).

Using the notion of the total energy $\mathcal E(t)$, then the energy identity (\ref{energy}) can be rewritten in a simpler form. 
Indeed, since $\frac{d}{dt} F(u(t)) =  f(u(t)) u_t(t)$, we have $\int_0^t  f(u)u_t d\t = F(u(t)) - F(u_0)$.
Therefore, the energy identity (\ref{energy}) is equivalent to 
\begin{align} \label{3-15}
\mathcal{E}(t)+\int^t_0\int_\Omega
g_1(u_t)u_tdxd\tau+\int^t_0\int_\Gamma g_2(w_t)w_td\Gamma
d\tau=\mathcal{E}(0),   \;\;  \text{for all}   \,\, t\in [0,T),
\end{align}
where $T$ is the maximal existence time. Taking the derivative with respect to $t$ gives
\begin{align}   \label{3-15de}
\mathcal{E}'(t)   +   \int_\Omega
g_1(u_t (t))u_t(t) \,dx +  \int_\Gamma g_2(w_t(t))w_t(t) d\Gamma = 0,    \;\;  \text{for all}   \,\, t\in [0,T).
\end{align}

Since the frictional damping terms satisfy $g_1(s)s\geq 0$ and $g_2(s)s\geq 0$ for all $s\in \mathbb R$,
we see from (\ref{3-15de}) that
\begin{align}   \label{negE'}
\mathcal E'(t) \leq 0,    \;\;  \text{for all}   \,\, t\in [0,T).
\end{align}
Therefore, $\mathcal E(t)$ is non-increasing for all $t \in [0,T)$.

\vspace{0.1 in}

\subsection{Main Results}
Our first result is the global existence of potential well solutions, provided the initial data belong to the set $\mathcal W_1$, 
which stands for the stable part of the potential well.

\begin{thm} [{\bf Global existence of potential well solutions}]
\label{thm3-1}
Assume that Assumption \ref{ass} and Assumption \ref{ass1}
hold. Let $1<p\leq 5$ and $q>1$. Assume further
$(u_0,w_0)\in\mathcal{W}_1$ and the total energy $\mathcal{E}(0)<d$. Then, system
\eqref{PDE} admits a global weak solution $(u,w)$. In addition,  for
any $t\geq 0$, the potential energy $\mathcal{J}(u(t),w(t))$, the total energy $\mathcal E(t)$ and the quadratic energy $E(t)$ satisfy
\begin{align}
\begin{cases}
    (i)\  \mathcal{J}(u(t),w(t))\leq
\mathcal{E}(t)\leq\mathcal{E}(0)<d,  \\
    (ii)\
(u(t),w(t))\in\mathcal{W}_1,\\
    (iii)\  E(t) < \displaystyle\frac{cd}{c-2}, \\
    (iv)\ \displaystyle\frac{c-2}{c}E(t)\leq\mathcal{E}(t)\leq E(t),
 \end{cases}\nonumber
\end{align}
 where
$c=\min\{p+1,q+1\}>2$.
\end{thm}

\begin{rmk}
There are two global existence results: in Theorem \ref{t:1}, the weak solution is global provided the damping terms dominate the source terms, i.e, 
$m\geq p$ and $r\geq q$; whereas, in Theorem \ref{thm3-1}, the existence of global solutions is assured by the assumption that the initial data belong to $\mathcal W_1$ and initial energy is sufficiently small.  
\end{rmk}

The following theorem establishes  the uniform decay rates of energy.

\begin{thm} [{\bf Energy decay rates}]\label{thm4-1}
Assume that Assumption \ref{ass} and Assumption \ref{ass1}
hold. Further assume that $1<p<5$, $k>1$, and $u_0\in L^{m+1}(\Omega)$. Also assume $(u_0,w_0)\in\tilde{\mathcal{W}}^{\delta}_1$ and $\mathcal{E}(0)\leq\Lambda(s^*-\delta)$,
for a sufficiently small $\delta>0$.
Moreover, assume $u\in
L^\infty(\mathbb{R}^+;L^{\frac{3}{2}(m-1)}(\Omega))$ if $m>5$.  Then
the global solution of problem \eqref{PDE} furnished by Theorem \ref{thm3-1}
has the following decay rates.

(i) If $g_1$ and $g_2$ are linearly bounded near the origin, then
the total energy $\mathcal{E}(t)$ and the quadratic energy $E(t)$ decay to zero exponentially, namely, for any $t\geq0$,
\begin{align}\label{exp}
\frac{c-2}{c}E(t)\leq\mathcal{E}(t)\leq \frac{C\mathcal{E}(0)}{e^{at}},
\end{align}
where $C$ and $a$ are positive constants independent of initial data.\\

(ii) If at least one of $g_1$ and $g_2$ are not linearly bounded near
the origin, then the total energy $\mathcal{E}(t)$ and the quadratic
energy $E(t)$ decay algebraically,
\begin{align}\label{alg}
\frac{c-2}{c}E(t)\leq\mathcal{E}(t)\leq \frac{C(\mathcal E(0))}{(1+t)^{b}},
\end{align}
where $b$ is given by \eqref{4-10}.
\end{thm}


\section{Global existence of potential well solutions}
This section is devoted to proving the global existence of potential well solutions if the initial data are in the stable set $\mathcal{W}_1$. 
In particular, we justify Theorem \ref{thm3-1} using the following argument.
\begin{proof}[Proof of Theorem \ref{thm3-1}]
Let the initial data $(u_0,w_0)\in\mathcal{W}_1$ and $\mathcal{E}(0)<d$.
The local well-posedness of a weak solution $(u(t),w(t))$ on $[0,T)$ is guaranteed by Theorem \ref{t:1} from paper \cite{Becklin-Rammaha2}, where $[0,T)$ is the maximal interval of existence.

We first show $(u(t),w(t))\in\mathcal{W}_1$ for all $t\in[0,T)$, namely, if the initial data belong to $\mathcal W_1$, 
the solution trajectory $(u(t),w(t))$ belongs to $\mathcal W_1$ during the entire life span of the solution. 

Because of \eqref{negE'}, we know $\mathcal E(t)$ is non-increasing as long as the solution exists. 
Therefore, $\mathcal{E}(t)\leq\mathcal{E}(0)<d$ since we assume the initial total energy is less than the depth $d$ of the potential well.
Also, the potential energy $\mathcal J(u(t),w(t))$ is not larger than the total energy $\mathcal E(t)$. This gives us that for any $t\in [0,T)$,
\begin{align}\label{3-16}
\mathcal{J}(u(t),w(t))\leq\mathcal{E}(t)\leq\mathcal{E}(0)<d.
\end{align}
Then part (i) is proved, and we obtain that $(u(t),w(t))\in\mathcal{W}$ for all $t\in [0,T)$.

To prove $(u(t),w(t))\in\mathcal{W}_1$ for all $t\in [0,T)$, we argue by contradiction.
We assume that there exists a time $t_1\in(0,T)$ such that
$(u(t_1),w(t_1))\notin\mathcal{W}_1$. 
But $(u(t_1),w(t_1)) \in \mathcal{W}$.
BY recalling
$\mathcal{W}_1\cup\mathcal{W}_2=\mathcal{W}$ and
$\mathcal{W}_1\cap\mathcal{W}_2=\emptyset$, then it must be the case
$(u(t_1),w(t_1))\in\mathcal{W}_2$.

Recall $F'(\xi) = f(\xi)$ and $|f(\xi)|\leq (p+1) M |\xi|^p$ for any $\xi \in \mathbb R$ due to \eqref{3-6}. Now, for any $t$, $t_0\in [0,T)$, then by the mean value theorem, one has,
\begin{align}   \label{Ftt0}
\int_\Omega|F(u(t))-F(u(t_0))|dx&\leq C\int_\Omega
(|u(t)|^p+|u(t_0)|^p)|u(t)-u(t_0)|dx \nonumber\\
&\leq C(\|u(t)\|^p_{p+1}+\|u(t_0)\|^p_{p+1})\|u(t)-u(t_0)\|_{p+1}   \notag\\
&\leq  C(\|u(t)\|^p_{1,\Omega}+\|u(t_0)\|^p_{1,\Omega})\|u(t)-u(t_0)\|_{1,\Omega},
\end{align}
where we have used the assumption that $p\leq 5$ and the embedding
$H^1_{\Gamma_0}(\Omega)\hookrightarrow L^6(\Omega)$. 
By the continuity of $u(t)$, namely, $u\in
C([0,T);H^1_{\Gamma_0}(\Omega))$, we can let $t$ approach $t_0$ in \eqref{Ftt0} to conclude that
$$
\int_\Omega F(u(t))\to \int_\Omega F(u(t_0)),\ \ \mbox{as}\ \ t\to
t_0,
$$
which implies that the function $t\mapsto\int_\Omega F(u(t))dx$ is
continuous on $[0,T)$.

Similarly, we can show that the function $t\mapsto\int_\Gamma
H(w(t))d\Gamma$ is continuous on $[0,T)$, by using $w \in C([0,T);H^2_0(\Gamma))$.

Therefore,  the mapping  
\begin{align}   \label{contt}
t \mapsto    \|\nabla u(t)\|^2_2+|\Delta w(t)|^2_2-(p+1)\int_\Omega
F(u(t))dx-(q+1)\int_\Gamma H(w(t))d\Gamma
\end{align}
is continuous.

In view of $(u(0),w(0))\in\mathcal{W}_1$ and
$(u(t_1),w(t_1))\in\mathcal{W}_2$ as well as the continuity of the function in (\ref{contt}), 
 the intermediate value theorem asserts that there exists a time $s\in (0,t_1)$ such that
\begin{align}\label{3-17}
\|\nabla u(s)\|^2_2+|\Delta w(s)|^2_2=(p+1)\int_\Omega
F(u(s))dx+(q+1)\int_\Gamma H(w(s))d\Gamma.
\end{align}
Define $t^*$ be the supremum of all $s\in(0,t_1)$ satisfying
\eqref{3-17}. 
Because of the continuity of the function in (\ref{contt}), we see that
$t^*\in(0,t_1)$ satisfying \eqref{3-17}, and $(u(t),w(t))\in\mathcal{W}_2$
for
any $t\in (t^*,t_1]$. We consider two cases:

\vspace{0.1 in}

\emph{Case 1.}  $(u(t^*),w(t^*))\neq(0,0)$. Since \eqref{3-17} holds
for $t^*$, then $(u(t^*),w(t^*))\in\mathcal{N}$, by the definition of the Nehari manifold $\mathcal N$  in (\ref{3-10}).
Then, we obtain from \eqref{depth} that $\mathcal{J}(u(t^*),w(t^*))\geq d$,  contradicting \eqref{3-16}.

\vspace{0.1 in}

\emph{Case 2.} $(u(t^*),w(t^*))=(0,0)$. Note that
$(u(t),w(t))\in\mathcal{W}_2$ for any $t\in(t^*,t_1]$. We conclude
from the definition of the set $\mathcal W_2$ and \eqref{3-5} that for any $t\in(t^*,t_1]$,
\begin{align}
\|\nabla u(t)\|^2_2+|\Delta w(t)|^2_2 <
C(\|u(t)\|^{p+1}_{p+1}+|w(t)|^{q+1}_{q+1})\leq C(\|\nabla
u(t)\|^{p+1}_2+|\Delta w(t)|^{q+1}_2),\nonumber
\end{align}
because $p\leq 5$. This implies that
\begin{align}   \label{uwX}
\|(u(t),w(t))\|^2_X<C(\|(u(t),w(t))\|^{p+1}_X+\|(u(t),w(t))\|^{q+1}_X),  \;\;  \text{for all}  \,\, t\in (t^*,t_1].
\end{align}
Since $(0,0)$ does not belong to $\mathcal W_2$, then $(u(t),w(t)) \not=(0,0)$ for any $t\in (t^*,t_1]$.
Then, we can divide both sides of (\ref{uwX}) by $\|(u(t),w(t))\|^2_X$ to obtain
$$
\|(u(t),w(t))\|^{p-1}_X+\|(u(t),w(t))\|^{q-1}_X>\frac{1}{C}.
$$
This yields  $\|(u(t),w(t))\|_X>s_0$, for any $t\in (t^*,t_1]$,
where $s_0>0$ is the unique positive solution of
$s^{p-1}+s^{q-1}=\frac{1}{C}$, where $p$,   $q>1$. 
Since the weak solution $(u(t),w(t))$ is continuous from $[0,T)$ to $X$, one has
$\|(u(t^*),w(t^*))\|_X\geq s_0>0$. This contradicts  the assumption
$(u(t^*),w(t^*))=(0,0)$. Therefore $(u(t),w(t))\in \mathcal{W}_1$ for all $t\in [0,T)$. Hence, claim (ii) is proved.

In the following, we prove that the weak solution $(u(t),w(t))$ on $[0,T)$ is global in time, i.e., the maximum lifespan $T=\infty$.  
To this end, we need to prove that the quadratic energy $E(t)$ has a uniform bound independent of time.

By \eqref{3-16}, the potential energy $\mathcal{J}(u(t),w(t))<d$ for all $t\in[0,T)$,
i.e.,
\begin{align}\label{3-18}
d>\mathcal{J}((u(t),w(t)))>\frac{1}{2}(\|\nabla u(t)\|^2_2+|\Delta
w(t)|^2_2)-\int_\Omega F(u(t))dx-\int_\Gamma H(w(t))d\Gamma.
\end{align}
Since we have proved that $(u(t),w(t))\in\mathcal{W}_1$ for all $t\in[0,T)$, then
\begin{align}\label{3-18-1}
\int_\Omega F(u(t))dx+\int_\Gamma H(w(t))d\Gamma\leq \frac{1}{c}(\|\nabla
u(t)\|^2_2+|\Delta w(t)|^2_2),
\end{align}
where $c=\min\{p+1,q+1\}>2$. It follows from \eqref{3-18} and
\eqref{3-18-1} that for any $t\in [0,T)$,
\begin{align}\label{3-19}
\int_\Omega F(u(t))dx+\int_\Gamma H(w(t))d\Gamma<\frac{2d}{c-2}.
\end{align}
Substituting \eqref{3-19} into \eqref{3-15}, we get that for any
$t\in[0,T),$
\begin{align}
&E(t)+\int^t_0\int_\Omega g_1(u_t)u_tdxd\tau+\int^t_0\int_\Gamma
g_2(w_t)w_t d\Gamma d\tau\nonumber\\
&\quad=\mathcal{E}(0)+\int_\Omega
F(u(t))dx+\int_\Gamma H(w(t))d\Gamma\nonumber\\
&\quad<d+\frac{2d}{c-2}=\frac{cd}{c-2}.\nonumber
\end{align}
Since $g_1(s)s\geq 0$ and $g_2(s)s\geq 0$ for all $s\in \mathbb R$, we obtain
\begin{align}   \label{ubE}
E(t) <   \frac{cd}{c-2}, \;\;\text{for all}\,\, t\in [0,T),
\end{align}
proving claim (iii).

Recall Theorem \ref{t:1} states that the local well-posedness of weak solutions defined on $[0,T_0]$ where $T_0$ 
depends on the initial quadratic energy $E(0)$. Due to (\ref{ubE}), $E(T_0)$ and $E(0)$ have the same upper bound. 
So we can extend the local solution from the time $T_0$ to $2T_0$. 
By iterating this procedure, one can obtain a global weak solution defined on $[0,\infty)$. That is to say, the maximum lifespan $T=\infty$.

To show (iv), we observe from \eqref{3-9} that $\mathcal{E}(t) \leq E(t)$ for all
$t\in[0,\infty)$ since $F$ and $H$ are nonnegative functions. Also, by (\ref{3-9}) and (\ref{3-18-1}), we obtain
$$
\mathcal{E}(t)\geq
\frac{1}{2}(\|u_t\|^2_2+|w_t|^2_2)+\left(\frac{1}{2}-\frac{1}{c}\right)(\|\nabla
u\|^2_2+|\Delta w|^2_2)\geq \frac{c-2}{c}E(t).
$$
The proof for Theorem \ref{thm3-1} is complete.
\end{proof}

\vspace{0.1 in}

\begin{rmk}
The continuity of the solution $(u,w)$ mapping from $[0,T)$ to $X=H^1_{\Gamma_0}(\Omega)\times H^2_0(\Gamma)$ is critical for the argument above, and essential for the validity of the entire paper. For instance, the implementation of the intermediate value theorem in the above proof depends on the continuity of the solution. Such regularity (namely, $u\in C([0,T);H^1_{\G_0}(\O))$ 
and $w\in C([0,T);H^2(\G))$) is guaranteed by the local well-posedness result (Theorem \ref{t:1}), proved by Becklin and Rammaha in \cite{Becklin-Rammaha2}. The method to prove the local existence for system \eqref{PDE} in \cite{Becklin-Rammaha2} consists of the theory of monotone operators and nonlinear semi-groups. Specifically, using Kato's Theorem (see, e.g. \cite{Sh}), the system has a solution $(u,w) \in W^{1,\infty}(0,T;X)$ if the source terms are subcritical. Then, the extension to supercritical source terms concludes that $(u,w) \in C([0,T);X)$. On the other hand, in another paper \cite{Becklin-Rammaha1}, Becklin and Rammaha studied a related model with restoring source terms but no damping terms, and by using Galerkin method, the local existence of weak solutions was shown but the solutions have only weak continuity in time. 
\end{rmk}

\vspace{0.1 in}

We end this section by giving a proof of this following lemma, which states that if the initial data belong to 
$\tilde{\mathcal{W}}^{\delta}_1$ (the closed subset of $\mathcal W_1$ constructed in subsection \ref{subsetW})
and if the initial total energy is sufficiently small, then the solution remains in $\tilde{\mathcal{W}}^{\delta}_1$ for all time. 

Recall the function $\Lambda(s)$ is defined in \eqref{Lambda} and $s^*$ is the location of the maximum of $\Lambda(s)$ on $\mathbb R^+$.
Also, recall the set $\tilde{\mathcal{W}}^{\delta}_1$ is defined in \eqref{3-14}.

\begin{lem}\label{lem3-3}
Suppose Assumption \ref{ass} and Assumption \ref{ass1} are valid. Let $1<p\leq 5$, $q>1$, and  $\delta>0$ is sufficiently small. Assume 
$\mathcal{E}(0)\leq\Lambda(s^*-\delta)$ and $(u_0,w_0)\in\tilde{\mathcal{W}}^{\delta}_1$. 
Then system \eqref{PDE} admits a global solution
$(u,v)$ satisfing $(u(t),w(t))\in\tilde{\mathcal{W}}^{\delta}_1$ for all $t\geq0$.

\end{lem}
\begin{proof}
Recall that, in subsection \ref{subsetW}, we have shown that the function $\Lambda(s)$ defined in \eqref{Lambda}
attains it maximum at $s=s^*$ over $\mathbb R^+$, and $\Lambda(s^*) \leq d$ due to \eqref{Ld}.
Since $\Lambda(t)$ is strictly increasing on $(0,s^*)$, we see that $\mathcal E(0) \leq \Lambda(s^*-\delta) < d$.
Also, since $(u_0,w_0)\in\tilde{\mathcal{W}}^{\delta}_1 \subset \mathcal W_1$, 
then thanks to Theorem \ref{thm3-1}, there exists a global solution $(u,v)$ with
$\mathcal{J}(u(t),w(t)) \leq   \mathcal E(0) \leq \Lambda(s^*-\delta)$ for all $t\geq0$. It
remains to show that $\|(u(t),w(t))\|_X\leq s^*-\delta$ for all $t\geq 0$. Since
$\|(u_0,w_0)\|_X\leq s^*-\delta$ and $(u,w)\in C(\mathbb{R}^+,X)$,
we assume to the contrary that there exists   $t_1>0$ such that
$\|(u(t_1),w(t_1))\|_X=s^*-\delta+\delta_0$ for some $\delta_0\in(0,\delta)$. By \eqref{3-13-1}, we have
$\mathcal{J}(u(t_1),w(t_1))\geq\Lambda(s^*-\delta+\delta_0)>\Lambda(s^*-\delta)$,
because $\Lambda(t)$ is strictly increasing on $(0,s^*)$.
This contradicts $\mathcal{J}(u(t),w(t))\leq\Lambda(s^*-\delta)$ for any $t\geq0$.
\end{proof}

\vspace{0.1 in}

\section{Energy Decay Rates}   \label{sec-decay}
This section is devoted to proving Theorem \ref{thm4-1}, namely, the uniform energy decay rates of potential well solutions.

First we remark that, if the solution decays, then at large time, the solution becomes ``small", so the behavior of damping terms near the origin 
determines the decay rates of solutions. 

In what  follows, we introduce some concave functions that capture the growth rates of damping terms near the origin.

\subsection{Concave functions that reflect the behavior of damping terms near the origin}  \label{concave}

Let $\phi_i:[0,\infty)\to[0,\infty)$ be concave,
 increasing, continuous functions vanishing at the origin, such
that, for $i=1,2,$
\begin{align}\label{4-1}
\phi_i(g_i(s)s)\geq |g_i(s)|^2+s^2,\ \ \mbox{for}\ |s|<1.
\end{align}
Also, we define a function $\Phi:[0,\infty)\to[0,\infty)$ by
\begin{align}\label{4-2}
\Phi(s):=\phi_1(s)+\phi_2(s)+s,\ \ s\geq0.
\end{align}
Note the function $\Phi$ is also concave, increasing, continuous and vanishing at the origin.

Now, we show that the concave functions $\phi_i$ ($i=1,2$) satisfying \eqref{4-1} can always be constructed. 
Indeed, recall that  $g_1$ and $g_2$ are continuous monotone
increasing functions passing the origin. If $g_1$ and $g_2$ are bounded above
and below by linear or super-linear functions near the origin, i.e.,
\begin{align}\label{4-3}
c_1|s|^m\leq |g_1(s)|\leq c_2|s|^m,\ \ c_3|s|^r\leq |g_2(s)|\leq
c_4|s|^r,\ \ \forall\ |s|<1,
\end{align}
where $m,r\geq 1$ and $c_i>0$, $i=1,2,3,4$, then we choose
\begin{align}\label{4-4}
\phi_1(s)=(c_1)^{-\frac{2}{m+1}}(1+c_2^2)s^{\frac{2}{m+1}},\ \
\phi_2(s)=(c_3)^{-\frac{2}{r+1}}(1+c_4^2)s^{\frac{2}{r+1}}.
\end{align}
To check that \eqref{4-4} satisfies \eqref{4-1}, we  calculate directly:
\begin{align}
\phi_1(g_1(s)s)& = (c_1)^{-\frac{2}{m+1}}(1+c_2^2)(g_1(s)s)^{\frac{2}{m+1}}
\geq c_1^{-\frac{2}{m+1}}(1+c_2^2)(c_1|s|^{m+1})^{\frac{2}{m+1}}\nonumber\\
&= (1+c_2^2)s^2\geq s^2+(c_2|s|^m)^2\geq s^2+|g_1(s)|^2,\
\mbox{for}\ \mbox{all}\ |s|<1.\nonumber
\end{align}
In particular, if $g_1$ and $g_2$ are both linearly bounded near the origin, 
then the explicit functions $\phi_1$ and $\phi_2$ given in \eqref{4-4} are both linear functions.

On the other hand, if $g_1$ and $g_2$ are bounded by sublinear functions near the origin, that is,
\begin{align}\label{4-5}
c_1|s|^{\kappa_1}\leq |g_1(s)|\leq c_2|s|^{\kappa_1},\ \
c_3|s|^{\kappa_2}\leq |g_2(s)|\leq c_4|s|^{\kappa_2},\ \ \mbox{for}\
\mbox{all}\ |s|<1,
\end{align}
where $0<\kappa_1,\kappa_2<1$ and  $c_i>0$ $(i=1,2,3,4)$. 
In this case, we can select
\begin{align}\label{4-6}
\phi_1(s)=(c_1)^{-\frac{2\kappa_1}{\kappa_1+1}}(1+c_2^2)s^{\frac{2\kappa_1}{\kappa_1+1}},\
\
\phi_2(s)=(c_3)^{-\frac{2\kappa_2}{\kappa_2+1}}(1+c_4^2)s^{\frac{2\kappa_2}{\kappa_2+1}}.
\end{align}

From \eqref{4-4} and \eqref{4-6}, we know that one can always construct $\phi_1$ and $\phi_2$ in the form
\begin{align}\label{4-7}
\phi_i(s)=C_is^{\nu_i},   \;\;i=1, 2,
\end{align}
for some constants $C_1$ and $C_2$, where
\begin{align} \label{4-8}
\nu_1=\frac{2}{m+1}\ \mbox{or}\ \frac{2\kappa_1}{\kappa_1+1},\ \
\nu_2=\frac{2}{r+1}\ \mbox{or}\ \frac{2\kappa_2}{\kappa_2+1},
\end{align}
depending on the growth rates of $g_1$ and $g_2$ near the origin, which are specified in \eqref{4-3} and \eqref{4-5}.

Define
\begin{align}\label{4-9}
\beta:=\max_{i=1,2}\left\{\frac{1}{\nu_i}\right\}.
\end{align}
Note that $\beta>1$ if at least one of $g_1$ and $g_2$ are not linearly bounded near the origin, and in this case we put
\begin{align}\label{4-10}
b:=(\beta-1)^{-1}>0.
\end{align}

\vspace{0.1 in}

\subsection{A stabilization estimate}

For convenience, we put:
\begin{align}    \label{Dt}
D(t):=\int^t_0\int_\Omega g_1(u_t)u_tdxd\tau+\int^t_0\int_\Gamma
g_2(w_t)w_td\Gamma d\tau.
\end{align}
Since $g_1(s)s\geq 0$ and $g_2(s) s \geq 0$, then $D(t)\geq
0$. Using this notation, the energy identity \eqref{3-15} can be written in the concise form:
\begin{align}\label{en-id}
\mathcal{E}(t)+D(t)=\mathcal{E}(0).
\end{align}
From (\ref{en-id}), we see that the behavior of damping terms determines the decay rates of the total energy $\mathcal{E}(t)$.

Define
\begin{align}\label{T0}
T_0:=\max\left\{1,\frac{1}{|\Omega|},\frac{1}{|\Gamma|},\frac{8c
c_0}{c-2}\right\},
\end{align}
where $c=\min\{p+1,q+1\}>2$ and $c_0>0$ is defined in (\ref{4-13-3}).

\begin{lem}\label{lem4-1}
Suppose that Assumption \ref{ass} and Assumption \ref{ass1}
hold. Assume that $1<p<5$, $k>1$, and $u_0\in L^{m+1}(\Omega)$.
Also assume $(u_0,w_0)\in\tilde{\mathcal{W}}^{\delta}_1$ and $\mathcal{E}(0)\leq\Lambda(s^*-\delta)$,
for a sufficiently small $\delta>0$.
In addition, suppose $u \in L^\infty(\mathbb{R}^+;L^{\frac{3}{2}(m-1)}(\Omega))$ if $m>5$.  Then
the global solution of system \eqref{PDE} furnished by Theorem \ref{thm3-1}
satisfies for all $T\geq T_0$,
\begin{align}\label{4-12}
\mathcal{E}(T)\leq\tilde{C}\Phi(D(T)),
\end{align}
where $T_0$ is given in \eqref{T0}, $\Phi$ is defined in
\eqref{4-2}, and $\tilde{C}>0$ defined in \eqref{Ctilde} is independent of $T$.
\end{lem}
\begin{proof}
Throughout the proof, we assume $T\geq T_0$, where $T_0$ is given in \eqref{T0}.
By the regularity of weak solutions specified in Definition \ref{def:weaksln}, 
we know that $u_t \in L^{m+1}(\Omega\times(0,T))$. 
Since we assume $u_0\in L^{m+1}(\Omega)$, then
the fundamental theorem of Calculus implies 
\begin{align}\label{4-12-1}
\int^T_0\int_\Omega|u|^{m+1}dxdt&= \int^T_0\int_\Omega\left|\int^t_0u_t(\tau)d\tau+u_0\right|^{m+1}dxdt\nonumber\\
&\leq C(T^{m+1}\|u_t(t)\|^{m+1}_{L^{m+1}(\Omega\times(0,T))}+T\|u_0\|^{m+1}_{m+1})<\infty.
\end{align}
This implies $u \in L^{m+1}(\Omega\times(0,T))$. We can use the same argument to get  
$w \in L^{r+1}(\Gamma\times(0,T))$. In this situation, we can replace
$\phi$ by $u$ in \eqref{wkslnwave} and $\psi$ by $w$ in \eqref{wkslnplt};
then adding the results yields
\begin{align}    \label{replace}
&\int_\Omega u_tudx\bigg|^T_0+\int_\Gamma(w_tw+\gamma
uw)d\Gamma\bigg|^T_0-\int^T_0(\|u_t\|^2_2+|w_t|^2_2)dt+\int^T_0(\|\nabla
u\|^2_2+|\Delta w|^2_2)dt\nonumber\\
&\qquad-2\int^T_0\int_\Gamma\gamma uw_td\Gamma
dt+\int^T_0\int_\Omega g_1(u_t)udxdt+\int^T_0\int_\Gamma
g_2(w_t)w d\Gamma dt\nonumber\\
&\qquad=\int^T_0\int_\Omega f(u)udxdt+\int^T_0\int_\Gamma
h(w)wd\Gamma dt. 
\end{align}
Multiply equality (\ref{replace}) by $1/2$, and use \eqref{qua} and \eqref{3-3}, we obtain
\begin{align}\label{4-13-1}
\int^T_0E(t)dt&=\underbrace{\int^T_0(\|u_t\|^2_2+|w_t|^2_2)dt}_{=I_3} \,\underbrace{-\frac{1}{2}\int_\Omega
u_tudx\bigg|^T_0-\frac{1}{2}\int_\Gamma (w_tw+\gamma u w)
d\Gamma\bigg|^T_0}_{=I_1}\nonumber\\
&\underbrace{+\int^T_0\int_\Gamma \gamma u w_td\Gamma
dt}_{=I_4} \, \underbrace{-\frac{1}{2}\int^T_0\int_\Omega
g_1(u_t)udxdt}_{=I_6} \, \underbrace{-\frac{1}{2}\int^T_0\int_\Gamma g_2(w_t)wd\Gamma
dt}_{=I_5}\nonumber\\
&\underbrace{+\frac{(p+1)}{2}\int^T_0\int_\Omega
F(u)dxdt+\frac{q+1}{2}\int^T_0\int_\Gamma H(w)d\Gamma dt}_{=I_2}.
\end{align}

In the sequel, we will estimate terms $I_i$  $(i=1,...,6)$ one by one. 

\subsubsection{\bf{Estimate for $I_1$}}

First, by using Cauchy-Schwarz inequality and Sobolev embedding theorem, we get
\begin{align}\label{4-13-2}
\left|\int_\Omega u_tudx+\int_\Gamma(w_tw+\gamma u
w)d\Gamma\right|&\leq\frac{1}{2}(\|u_t\|^2_2+|w_t|^2_2+\|u\|^2_2+2|w|^2_2+|\gamma
u|^2_2)\nonumber\\
&\leq\frac{1}{2}\Big[\|u_t\|^2_2+|w_t|^2_2+(c_1^*+c_*)\|\nabla
u\|^2_2+2c_2^*|\Delta w|^2_2\Big]\nonumber\\
&\leq c_0E(t),
\end{align}
where
\begin{align}\label{4-13-3}
c_0=\max\{1,(c_1^*+c_*),2c_2^*\},
\end{align}
and $c_1^*>0$ is the embedding constant of $\|u\|^2_2\leq
c_1^*\|\nabla u\|^2_2$, $c_2^*>0$ is the embedding constant of
$|w|^2_2\leq c_2^*|\Delta w|^2_2$ and $c_*>0$ is the embedding
constant of $|\gamma u|^2_2\leq c_*\|\nabla u\|^2_2$.

It follows from claim (iv) of Theorem \ref{thm3-1}, \eqref{4-13-2} and
\eqref{en-id} that
\begin{align}\label{4-14}
I_1\leq \frac{c_0}{2}(E(T)+E(0))\leq
\frac{cc_0}{2(c-2)}(\mathcal{E}(T)+\mathcal{E}(0))\leq
\frac{cc_0}{2(c-2)}(2\mathcal{E}(T)+D(T)).
\end{align}

\vspace{0.1 in}

\subsubsection{\bf{Estimate for $I_2$}}   \label{est-I2}
It follows from \eqref{3-5}, \eqref{normX} and \eqref{3-12} that
\begin{align}\label{4-15}
I_2&\leq \frac{1}{2}M(p+1)\int^T_0\|u\|^{p+1}_{p+1}dt+\frac{1}{2}M(q+1)\int^T_0|w|^{q+1}_{q+1}dt\nonumber\\
&\leq \frac{1}{2}M(p+1)K_1\int^T_0\|\nabla u\|^{p+1}_{2}dt+\frac{1}{2}M(q+1)K_2\int^T_0|\Delta w|^{q+1}_{2}dt\nonumber\\
&\leq \frac{1}{2}M(p+1)K_1\int^T_0  \|(u,w)\|^{p+1}_{X}dt +\frac{1}{2}M(q+1)K_2 \|(u,w)\|^{q+1}_{X}dt\nonumber\\
&=\frac{1}{2} \int^T_0\|(u,w)\|^{2}_{X}[M(p+1)K_1\|(u,w)\|^{p-1}_{X}+M(q+1)K_2\|(u,w)\|^{q-1}_{X}]dt.
\end{align}

Since $(u_0,w_0)\in\tilde{\mathcal{W}}^{\delta}_1$ and $\mathcal{E}(0)\leq\Lambda(s^*-\delta)$
and thanks to Lemma \ref{lem3-3}, we have $(u(t),w(t))\in \tilde{\mathcal{W}}^{\delta}_1$ for all $t\geq0$.
Then due to the definition of $\tilde{\mathcal{W}}^{\delta}_1$, i.e., formula (\ref{3-14}), 
we know that 
\begin{align}   \label{b-del}
\|(u(t),w(t))\|_X\leq s^*-\delta, \;\; \text{for all} \,\, t\geq 0. 
\end{align}

Then, we obtain from \eqref{4-15} and \eqref{b-del} that
\begin{align}\label{4-16}
I_2& \leq \frac{1}{2}\int^T_0\|(u,w)\|^{2}_{X}[M(p+1)K_1(s^* - \delta)^{p-1}+M(q+1)K_2(s^* - \delta)^{q-1}]dt\nonumber\\
&=\frac{1}{2}\xi \int^T_0\|(u,w)\|^{2}_{X}\leq \xi\int^T_0E(t)dt,
\end{align}
where the constant $\xi$ is defined as
\begin{align}  \label{def-xi}
\xi := M(p+1)K_1(s^* - \delta)^{p-1}+M(q+1)K_2(s^* - \delta)^{q-1} <1.
\end{align}
The fact that $\xi <1$ is because of (\ref{3-13-2}).

In sum, we conclude that there exists a constant $0<\xi<1$ such that
\begin{align}\label{4-17}
I_2\leq \xi\int^T_0E(t) dt.
\end{align}
We stress that the fact that $\xi$ is strictly less than 1 is crucial for our argument. 
Because, the right-hand side of \eqref{4-17} can be completely absorbed by the term $\int_0^T E(t) dt$ on the left-hand side of \eqref{4-13-1}. 
This makes the proof concise because there are no lower-order terms appearing in the stabilization estimate.

\vspace{0.1 in}

\subsubsection{\bf{Estimate for $I_3$}}
We define
\begin{align}  \label{AO}
A_\Omega:=\{(x,t)\in \Omega\times (0,T):|u_t(x,t)|<1\},
\end{align}
and
\begin{align}   \label{BO}
B_\Omega:=\{(x,t)\in \Omega\times
(0,T):|u_t(x,t)|\geq1\}.
\end{align}
From Assumption \ref{ass}, we infer that
$$
\alpha|s|^2\leq\alpha|s|^{m+1}\leq g_1(s)s,\ \ \forall\ |s|\geq1,
$$
since $m\geq 1$.

Now we use the concave functions $\phi_1$ and $\phi_2$ constructed in subsection \ref{concave}. 
Recall that $\phi_1$ and $\phi_2$ are related to the growth rates of $g_1$ and $g_2$ near the origin, respectively. 

By \eqref{4-1} and noting that $\phi_1$ maps $[0,\infty)$ to $[0,\infty)$, we derive that
\begin{align}\label{4-20-1}
\int^T_0\|u_t(t)\|^2_2dx&=\int_{A_\Omega}|u_t|^2dxdt+\int_{B_\Omega}|u_t|^2dxdt\nonumber\\
&\leq \int_{A_\Omega}\phi_1(g_1(u_t)u_t)dxdt+\frac{1}{\alpha}\int_{B_\Omega}g_1(u_t)u_tdxdt\nonumber\\
&\leq \int_0^T \int_{\Omega}   \phi_1(g_1(u_t)u_t)dxdt
+\frac{1}{\alpha}   \int_0^T \int_{\Omega}  g_1(u_t)u_tdxdt.
\end{align}
Since $\phi_1$ is concave, we can use Jensen's inequality to obtain
\begin{align}  \label{4-20-2}
\frac{1}{T|\Omega|}\int_0^T \int_{\Omega}   \phi_1(g_1(u_t)u_t)dxdt 
&\leq \phi_1\left(   \frac{1}{T|\Omega|}   \int^T_0\int_\Omega g_1(u_t)u_tdxdt\right)   \notag\\
&\leq  \phi_1\left(   \int^T_0\int_\Omega g_1(u_t)u_tdxdt\right),
\end{align}
where we have used the fact that $\phi$ is increasing and  $T|\Omega| \geq 1$ 
because $T\geq T_0 \geq \frac{1}{|\Omega|}$ from \eqref{T0}.

Combining \eqref{4-20-1} and \eqref{4-20-2} yields
\begin{align}\label{4-20}
\int^T_0\|u_t(t)\|^2_2dx \leq T|\Omega|\phi_1\left(\int^T_0\int_\Omega
g_1(u_t)u_tdxdt\right)+\frac{1}{\alpha}\int^T_0\int_\Omega g_1(u_t)u_t dxdt.
\end{align}

In the same manner, we can show 
\begin{align}\label{4-21}
\int^T_0|w_t(t)|^2_2 dt\leq T|\Gamma|\phi_2\left(\int^T_0\int_\Gamma
g_2(w_t)w_td\Gamma dt\right)+\frac{1}{\alpha}\int^T_0\int_\Gamma
g_2(w_t)w_td\Gamma dt.
\end{align}
Then it follows from \eqref{4-20} and \eqref{4-21} that
\begin{align}\label{4-22}
I_3&\leq T|\Omega|\phi_1\left(\int^T_0\int_\Omega
g_1(u_t)u_tdxdt\right)+T|\Gamma|\phi_2\left(\int^T_0\int_\Gamma
g_2(w_t)w_td\Gamma dt\right)\nonumber\\
&\quad+\frac{1}{\alpha}\left[\int^T_0\int_\Omega
g_1(u_t)u_tdxdt+\int^T_0\int_\Gamma g_2(w_t)w_td\Gamma dt\right].
\end{align}

\vspace{0.1 in}

\subsubsection{\bf{Estimate for $I_4$}}

By using Cauchy-Schwarz inequality, the embedding $|\gamma u|^2_2\leq c_*\|\nabla u\|^2_2$, and Young's inequality, we obtain
\begin{align*}
I_4  \leq   \int_0^T |\gamma u|_2 |w_t|_2  dt   \leq c_*^{1/2}   \int_0^T \|\nabla u\|_2 |w_t|_2  dt
\leq \frac{\varepsilon}{2} \int^T_0\|\nabla u\|^2_2dt+C_{\varepsilon}\int^T_0|w_t|^2_2dt,
\end{align*}
and together with \eqref{4-21}, we obtain
\begin{align}\label{4-23}
I_4&\leq \varepsilon\int^T_0E(t)dt
+T |\Gamma|  C_{\varepsilon} \phi_2\left(\int^T_0\int_\Gamma
g_2(w_t)w_td\Gamma dt\right)+\frac{C_\varepsilon}{\alpha}\int^T_0\int_\Gamma
g_2(w_t)w_td\Gamma dt.
\end{align}

\vspace{0.1 in}

\subsubsection{\bf{Estimate for $I_5$}}

To estimate $I_5$, as $A_\Omega$ and $B_\Omega$, we define
$$
A_\Gamma:=\{(x,t)\in \Gamma\times (0,T):|w_t(x,t)|<1\},
$$
and
$$B_\Gamma:=\{(x,t)\in \Gamma\times
(0,T):|w_t(x,t)|\geq1\}.
$$
By using H\"{o}lder's inequality, Young's inequality and
the definition of $E(t)$, we get that for any $\varepsilon>0$,
\begin{align}\label{4-24}
&\int^T_0\int_\Gamma|g_2(w_t)w|d\Gamma
dt =\int_{A_\Gamma}|g_2(w_t)w|d\Gamma
dt+\int_{B_\Gamma}|g_2(w_t)w|d\Gamma dt\nonumber\\
&\qquad\leq\left(\int^T_0|w|^2_2dt\right)^{\frac{1}{2}}\left(\int_{A_\Gamma}|g_2(w_t)|^2d\Gamma
dt\right)^{\frac{1}{2}}+\int_{B_\Gamma}|g_2(w_t)w|d\Gamma dt\nonumber\\
&\qquad\leq\varepsilon\int^T_0E(t)dt+C_\varepsilon\int_{A_\Gamma}|g_2(w_t)|^2d\Gamma
dt+\int_{B_\Gamma}|g_2(w_t)w|d\Gamma dt,
\end{align}
where we have used the Poincar\'e inequality $|w|_2 \leq C|\Delta w|_2^2 \leq C E(t)$.

Since $T\geq T_0 \geq  \frac{1}{|\Gamma|}$ from \eqref{T0}, we have $T |\Gamma| \geq 1$.
Also, recall the function $\phi_2 : [0,\infty) \rightarrow [0,\infty)$ is concave. 
Then, we can use Jensen's inequality and \eqref{4-1} to deduce
\begin{align}\label{4-25}
\int_{A_\Gamma}|g_2(w_t)|^2d\Gamma dt\leq
\int_{A_\Gamma}\phi_2(g_2(w_t)w_t)d\Gamma dt\leq
T|\Gamma|\phi_2\left(\int^T_0\int_\Gamma g_2(w_t)w_td\Gamma
dt\right).
\end{align}
Recalling Assumption \ref{ass}, we have $|g_2(s)|\leq\beta|s|^r$ for all $|s|\geq1$. Then H\"{o}lder's inequality implies
\begin{align}\label{4-26}
\int_{B_\Gamma}|g_2(w_t)w|d\Gamma dt &\leq\left(\int_{B_\Gamma}|w|^{r+1}d\Gamma
dt\right)^{\frac{1}{r+1}}\left(\int_{B_\Gamma}|g_2(w_t)|^{\frac{r+1}{r}}d\Gamma
dt\right)^{\frac{r}{r+1}}\nonumber\\
&\leq\left(\int^T_0|w|^{r+1}_{r+1}dt\right)^{\frac{1}{r+1}}\left(\int_{B_\Gamma}|g_2(w_t)||g_2(w_t)|^{\frac{1}{r}}d\Gamma
dt\right)^{\frac{r}{r+1}}\nonumber\\
&\leq\beta^{\frac{1}{r+1}}\left(\int^T_0|w|^{r+1}_{r+1}dt\right)^{\frac{1}{r+1}}\left(\int_{B_\Gamma}|g_2(w_t)|w_td\Gamma
dt\right)^{\frac{r}{r+1}}.
\end{align}
Recall that the claim (iii) of Theorem \ref{thm3-1} tells us $E(t)<\frac{cd}{c-2}$ for all $t\geq 0$. 
Also, Sobolev embedding shows $|w|_{r+1} \leq C|\Delta w|_2$. Hence, 
\begin{align}  \label{4-26-1}
\int^T_0|w|^{r+1}_{r+1}dt\leq C\int^T_0|\Delta w|^{r+1}_2dt\leq
C\int^T_0E^{\frac{r+1}{2}}(t)dt\leq C(c,d,r)\int^T_0E(t)dt,
\end{align}
since $r\geq 1$.

Combining \eqref{4-26} and \eqref{4-26-1}, and using Young's inequality, we obtain that for any $\varepsilon>0$,
\begin{align}\label{4-27}
\int_{B_\Gamma}|g_2(w_t)w|d\Gamma
dt&\leq C\left(\int^T_0E(t)dt\right)^{\frac{1}{r+1}}\left(\int^T_0\int_\Gamma|g_2(w_t)w_t|d\Gamma
dt\right)^{\frac{r}{r+1}}\nonumber\\
&\leq \varepsilon\int^T_0E(t)dt+C_\varepsilon\int^T_0\int_\Gamma|g_2(w_t)w_t|d\Gamma
dt.
\end{align}
Substituting \eqref{4-25} and \eqref{4-27} into \eqref{4-24}, we get
for any $\varepsilon>0$,
\begin{align}\label{4-28}
I_5 = \frac{1}{2} \int^T_0\int_\Gamma |g_2(w_t)w|d\Gamma dt&\leq
\varepsilon\int^T_0E(t)dt+C_\varepsilon
T|\Gamma|\phi_2\left(\int^T_0\int_\Gamma g_2(w_t)w_td\Gamma
dt\right)\nonumber\\
&\quad+C_\varepsilon\int^T_0\int_\Gamma g_2(w_t)w_td\Gamma dt.
\end{align}

\vspace{0.1 in}

\subsubsection{\bf{Estimate for $I_6$}}
Recall the sets $A_{\Omega}$ and $B_{\Omega}$ are defined in (\ref{AO})-(\ref{BO}).
Using H\"older's inequality and Young's inequality, we obtain
\begin{align}\label{4-29}
&\int^T_0\int_\Omega|g_1(u_t)u|dx
dt =\int_{A_\Omega}|g_1(u_t)u|dx
dt+\int_{B_\Omega}|g_1(u_t)u|dx dt\nonumber\\
&\qquad\leq\left(\int^T_0\|u\|^2_2dt\right)^{\frac{1}{2}}\left(\int_{A_\Omega}|g_1(u_t)|^2dx
dt\right)^{\frac{1}{2}}+\int_{B_\Omega}|g_1(u_t)u|dx dt\nonumber\\
&\qquad\leq\varepsilon\int^T_0E(t)dt+C_\varepsilon\int_{A_\Omega}|g_1(u_t)|^2dx
dt+\int_{B_\Omega}|g_1(u_t)u|dx dt.
\end{align}
for any $\varepsilon>0$.

Using \eqref{4-1} and Jensen's inequality, we obtain
\begin{align}\label{4-29-1}
\int_{A_\Omega}|g_1(u_t)|^2dx
dt\leq\int_{A_\Omega}\phi_1(g_1(u_t)u_t)dxdt\leq
T|\Omega|\phi_1\left(\int^T_0\int_\Omega g_1(u_t)u_tdxdt\right),
\end{align}
due to $T|\Omega|\geq1$.

We consider two cases to estimate the last term on the right-hand
side of \eqref{4-29}.\\
 \textbf{Case 1.}
$m\leq5.$\\
It follows from Assumption \ref{ass} that
$|g_1(s)|\leq\beta|s|^m\leq\beta|s|^5$ for $|s|\geq1$. Then we apply H\"{o}lder's inequality to deduce
\begin{align}
\int_{B_\Omega}|g_1(u_t)u|dxdt
&\leq\left(\int_{B_\Omega}|u|^6 dxdt\right)^{\frac{1}{6}}\left(\int_{B_\Omega}|g_1(u_t)|^{\frac{6}{5}} dx dt \right)^{\frac{5}{6}}\nonumber\\
&\leq\left(\int^T_0\|u\|^6_6dt\right)^{\frac{1}{6}}\left(\int_{B_\Omega}|g_1(u_t)||g_1(u_t)|^{\frac{1}{5}}  dx dt \right)^{\frac{5}{6}}\nonumber\\
&\leq\beta^{\frac{1}{6}}\left(\int^T_0\|u\|^6_6dt\right)^{\frac{1}{6}}\left(\int_{B_\Omega}|g_1(u_t)||u_t|dxdt\right)^{\frac{5}{6}},\nonumber
\end{align}
which along with
$$
\int^T_0\|u\|^6_6dt\leq C\int^T_0\|\nabla u\|^6_2dt\leq
C\int^T_0E^3(t)dt\leq C(c,d) \int^T_0E(t)dt,
$$
yields that 
\begin{align}\label{4-30}
\int_{B_\Omega}|g_1(u_t)u|dxdt&\leq C\left(\int^T_0E(t)dt\right)^{\frac{1}{6}}\left(\int_{B_\Omega}g_1(u_t)u_tdxdt\right)^{\frac{5}{6}}\nonumber\\
&\leq\varepsilon\int^T_0E(t)dt+C_\varepsilon\int^T_0\int_\Omega
g_1(u_t)u_t dxdt,
\end{align}
for any $\varepsilon>0$.

Combining \eqref{4-29-1} and \eqref{4-30} with \eqref{4-29}, we
obtain that for any $\varepsilon>0$,
\begin{align}\label{4-30-1}
I_6 = \frac{1}{2}\int^T_0\int_\Omega|g_1(u_t)u|dxdt&\leq \varepsilon\int^T_0E(t)dt+C_\varepsilon
T|\Omega|\phi_1\left(\int^T_0\int_\Omega
g_1(u_t)u_tdxdt\right)\nonumber\\
&\quad+C_\varepsilon\int^T_0\int_\Omega g_1(u_t)u_tdxdt,\ \ \ \mbox{if}\
\ m\leq 5.
\end{align}

\noindent
\textbf{Case 2.} $m>5$.

In this case, we assume an extra regularity on $u$, namely, $u \in L^{\infty}(\mathbb R^+; L^{\frac{3}{2}(m+1)}(\Omega))$. 

H\"{o}lder's inequality implies
\begin{align}\label{4-30-2}
\int_{B_\Omega}|g_1(u_t)u|dxdt\leq
\left(\int_{B_\Omega}|g_1(u_t)|^{\frac{m+1}{m}}dxdt\right)^{\frac{m}{m+1}}\left(\int_{B_\Omega}|u|^{m+1}dxdt\right)^{\frac{1}{m+1}}.
\end{align}
Since $|g_1(s)|\leq\beta |s|^m$ for $|s|\geq 1$, we have
\begin{align}\label{4-30-3}
\int_{B_\Omega}|g_1(u_t)|^{\frac{m+1}{m}}dxdt=\int_{B_\Omega}|g_1(u_t)||g_1(u_t)|^{\frac{1}{m}}dxdt\leq\beta^{\frac{1}{m}}\int_{B_\Omega}g_1(u_t)u_tdxdt.
\end{align}
Moreover, by H\"{o}lder's inequality and Sobolev embedding, we get
\begin{align}\label{4-30-4}
\int_{B_\Omega}|u|^{m+1}dxdt&\leq \int^T_0\|u\|^2_6\|u\|^{m-1}_{\frac{3}{2}(m-1)}dt  
\leq \int^T_0\|\nabla u\|^2_2\|u\|^{m-1}_{\frac{3}{2}(m-1)}dt\nonumber\\
&\leq C\|u\|^{m-1}_{L^\infty(\mathbb{R}^+;L^{\frac{3}{2}(m-1)}(\Omega))}\int^T_0E(t)dt.
\end{align}
Substituting \eqref{4-30-3} and \eqref{4-30-4} into \eqref{4-30-2}
and using Young's inequality, we obtain that for any
$\varepsilon>0$,
\begin{align}\label{4-30-5}
\int_{B_\Omega}|g_1(u_t)u|dxdt\leq\varepsilon\|u\|^{m-1}_{L^\infty(\mathbb{R}^+;L^{\frac{3}{2}(m-1)}(\Omega))}\int^T_0E(t)dt
+C_\varepsilon\int^T_0\int_\Omega g_1(u_t)u_tdxdt.
\end{align}
Inserting \eqref{4-29-1} and \eqref{4-30-5} into \eqref{4-29}, we get that for any $\varepsilon>0$,
\begin{align}\label{4-31}
&I_6 = \frac{1}{2}\int_0^T  \int_{\Omega} |g_1(u_t)u|dxdt \nonumber\\
&\qquad\leq\varepsilon \left(1+\|u\|^{m-1}_{L^\infty(\mathbb{R}^+;L^{\frac{3}{2}(m-1)}(\Omega))}\right)\int^T_0E(t)dt+C_\varepsilon
T|\Omega|\phi_1\left(\int^T_0\int_\Omega
g_1(u_t)u_tdxdt\right)\nonumber\\
&\qquad \quad+C_\varepsilon\int^T_0\int_\Omega g_1(u_t)u_tdxdt,\ \
\mbox{if}\ \ m>5.
\end{align}

We have finished estimating all terms $I_i$ $(i=1,...,6)$ from the right-hand side of equality \eqref{4-13-1}.

Recalling that, if $m>5$, we assume an extra regularity assumption on $u$ that $u\in
L^\infty(\mathbb{R}^+;L^{\frac{3}{2}(m-1)}(\Omega))$. 
Then, by taking $\varepsilon>0$ small enough, we can apply estimates \eqref{4-14}, \eqref{4-17}, \eqref{4-22}, \eqref{4-23},  \eqref{4-28},
\eqref{4-30-1} and \eqref{4-31} to (\ref{4-13-1}). It follows that
\begin{align}\label{4-32}
\frac{1-\xi}{2}\int^T_0E(t)dt\leq \frac{cc_0}{c-2}(2\mathcal{E}(T)+D(T))+C(\varepsilon,\alpha) (|\Omega| + |\Gamma|)T\Phi(D(T)),
\end{align}
where $T\geq T_0\geq1$. Note here the concave function $\Phi(s):=\phi_1(s)+ \phi_2(s) +s$,
and $D(T)$ is defined by \eqref{Dt}. Also, the constant $0<\xi<1$ is defined by \eqref{def-xi}.

 Noting that $\mathcal{E}(t)$ is non-increasing in time and $\mathcal{E}(t)\leq E(t)$ for any $t\geq0$, we have
\begin{align}  \label{TET}
T\mathcal{E}(T)\leq\int^T_0\mathcal{E}(t)dt\leq\int^T_0E(t)dt. 
\end{align}

Because of \eqref{4-32} and \eqref{TET}, we see that
\begin{align*}
\Big(\frac{1-\xi}{2}   T   -    \frac{2cc_0}{c-2}\Big)   \mathcal{E}(T)  \leq  \frac{c c_0}{c-2}D(T)+C(\varepsilon,\alpha) \Big(|\Omega| + |\Gamma| \Big)T\Phi(D(T)).\end{align*}

Since $T\geq T_0\geq \frac{8c c_0}{(c-2)(1-\xi)}$, we obtain
\begin{align}\label{4-33}
\frac{1-\xi}{4}T\mathcal{E}(T)\leq  \frac{c c_0}{c-2}D(T)+C(\varepsilon,\alpha) \Big(|\Omega| + |\Gamma|\Big)T\Phi(D(T)).
\end{align}

Because $T\geq T_0\geq 1$, then $\frac{1}{T}\leq 1$, we infer from \eqref{4-33} that
\begin{align}\label{4-34}
\frac{1-\xi}{4}\mathcal{E}(T)\leq \frac{c c_0}{c-2}D(T)+C(\varepsilon,\alpha) \Big(|\Omega| + |\Gamma| \Big)\Phi(D(T)).
\end{align}
Define the following constant $\tilde C$ independent of $T$:
\begin{align} \label{Ctilde}
\tilde{C}:=\frac{4}{1-\xi}\left[\frac{c
c_0}{c-2}+C(\varepsilon,\alpha) (|\Omega| + |\Gamma|)\right].
\end{align}
Then we get from \eqref{4-34} that for any $T\geq T_0$,
$$
\mathcal{E}(T)\leq\tilde{C}\Phi(D(T)).
$$
This completes the proof of Lemma \ref{lem4-1}.
\end{proof}

\vspace{0.1 in}

\subsection{Completion of the proof of Theorem \ref{thm4-1}}

Now we use the stabilization estimate provided by Lemma \ref{lem4-1}
to prove Theorem \ref{thm4-1}. The strategy of the proof is adopted from paper \cite{LTa} by Lasiecka and Tataru.
The idea is to relate the stabilization estimate to an ODE, and the decay rate of the solution of the ODE determines the energy decay rate of the PDE.

\begin{proof}[Proof of Theorem \ref{thm4-1}]
Let us fix a time $T\geq T_0$. From  \eqref{4-12}, we know that 
$$
\mathcal{E}(T) \leq \tilde{C} \Phi(D(T)).
$$
Note that $\mathcal{E}(T)+D(T)=\mathcal{E}(0)$. 
We define a concave function $\tilde \Phi: [0,\infty)  \rightarrow [0,\infty) $ by 
\begin{align}   \label{tildePhi}
\tilde \Phi(s) :=\tilde{C} \Phi(s) = \tilde{C} (\phi_1(s)+ \phi_2(s) + s).
\end{align}
Notice that $\tilde \Phi$ is concave, increasing, continuous and satisfying $\tilde \Phi(0)=0$.
Then
\begin{align}\label{4-69}
\mathcal{E}(T)\leq \tilde \Phi(D(T))=\tilde \Phi(\mathcal{E}(0)-\mathcal{E}(T)).
\end{align}
Hence, \eqref{4-69} yields
\begin{align}\label{4-70}
(I  +   \tilde \Phi^{-1}) \mathcal{E}(T) \leq \mathcal{E}(0).
\end{align}
Here, $\tilde \Phi^{-1}$ is convex, increasing, continuous and vanishing at the origin.

Recall that $(u(t),w(t)) \in \tilde {\mathcal W}_1^{\delta}$ and $\mathcal E(t) \leq \mathcal E(0) \leq \Lambda(s^* -\delta)$ for all $t\geq 0$.
Therefore, we can iterate \eqref{4-70} on $[mT,(m+1)T]$, $m=0,1,2,...,$ to obtain
$$
 (I  +   \tilde \Phi^{-1})  \mathcal{E}((m+1)T) \leq\mathcal{E}(mT),\
\ m=0,1,2,....
$$
Following \cite[Lemma 3.3]{LTa}, we have
\begin{align}\label{4-71}
\mathcal{E}(mT)\leq \sigma(m),\ \ m=0,1,2,3....
\end{align}
Here $\sigma(t)$ is the solution of the ODE
\begin{align}\label{4-72}
\begin{cases}
    \sigma'(t)+[I-(I+\tilde \Phi^{-1})^{-1}]\sigma(t)=0, \\
    \sigma(0)=\mathcal{E}(0).
 \end{cases}
\end{align}
Since $I-(I+\tilde \Phi^{-1})^{-1}=(I+\tilde \Phi)^{-1}$, we can reduce
\eqref{4-72} to
\begin{align}\label{4-73}
\begin{cases}
    \sigma'(t)+(I+\tilde \Phi)^{-1}\sigma(t)=0, \\
    \sigma(0)=\mathcal{E}(0),
 \end{cases}
\end{align}
where \eqref{4-73} has a unique solution on $[0,\infty)$. Noting
that $\tilde \Phi$ is increasing and vanishing at the origin, then
$(I+\tilde \Phi)^{-1}$ is also increasing and vanishing at the origin. Rewrite
\eqref{4-73} as $\sigma'(t)=-(I+\tilde \Phi)^{-1}\sigma(t)$ to obtain
$\sigma(t)$ is decreasing and approaching zero from above as $t\to\infty$.

For any $0<T<t$, there is an $m\in\mathbb{N}$ such that
$t=mT+\delta$, $0\leq\delta<T$, hence
$m=\frac{t}{T}-\frac{\delta}{T}>\frac{t}{T}-1$.  Since
$\mathcal{E}(t)$ and $\sigma(t)$ are monotone decreasing, we can
infer from \eqref{4-71} that for any $t>T$,
\begin{align}\label{4-74}
\mathcal{E}(t)=\mathcal{E}(mT+\delta)\leq\mathcal{E}(mT)\leq
\sigma(m)\leq \sigma\left(\frac{t}{T}-1\right).
\end{align}

\vspace{0.02 in}

(i) If $g_1$ and $g_2$ are linearly bounded near the origin, we see from \eqref{4-4} that $\phi_1$ and $\phi_2$ are linear, and
hence $\tilde \Phi$ is also linear because of (\ref{tildePhi}). Therefore $(I+\tilde \Phi)^{-1}$ is
a linear function. Then, \eqref{4-73} implies that for a constant $\gamma>0$,
\begin{align}
\begin{cases}
    \sigma'(t)+\gamma \sigma(t)=0, \\
    \sigma(0)=\mathcal{E}(0),
\end{cases}\nonumber
\end{align}
Then one has
$$
\sigma(t)=\mathcal{E}(0)e^{-\gamma t}.
$$
By using \eqref{4-74}, we obtain that for any $t>T$,
\begin{align}\label{4-75}
\mathcal{E}(t)\leq
\mathcal{E}(0)e^{-\gamma(\frac{t}{T}-1)}= e^{\gamma}\mathcal{E}(0) e^{-\frac{\gamma}{T}t}.
\end{align}
Putting $a=\frac{\gamma}{T}$, we get \eqref{exp}.\\

\vspace{0.02 in}

(ii) If at least one of $g_1$  and $g_2$ are not linearly bounded
near the origin. By \eqref{4-7}  we can select
$\phi_i(s)=C_is^{\nu_i}$, where $\nu_i \in (0,1)$ is given in \eqref{4-8}.

Note that if $\lambda=(I+\tilde \Phi)^{-1}(s)$ for $s\geq0$, then
$\lambda>0$. Moreover, for any $0\leq\lambda\leq1$,
\begin{align}
s&=(I+\tilde \Phi)\lambda=\lambda+\tilde{C} \left(\phi_1(\lambda)+  \phi_2(\lambda)  +   \lambda\right) \nonumber\\
&\leq C \left( \phi_1(\lambda)+  \phi_2(\lambda)  +    \lambda\right)\leq
C\lambda^{\min\{\nu_1, \nu_2 \}}. \nonumber
\end{align}
Then   there exists $C_0>0$ such that $\lambda\geq C_0s^{\beta}$
for any $0\leq\lambda\leq1$ where $\beta =  \max\{   \frac{1}{\nu_1},  \frac{1}{\nu_2}\}>1$, namely,
\begin{align}\label{4-76}
(I+\tilde \Phi)^{-1}(s)\geq C_0s^{\beta}\ \ \mbox{if}\ \
0\leq(I+\tilde \Phi)^{-1}(s)\leq1.
\end{align}
Noting that $(I+\tilde \Phi)^{-1}(\sigma(t))$ is decreasing to zero as
$t\to\infty$, then there exists $t_0\geq0$ such that
$(I+\tilde \Phi)^{-1}(\sigma(t))\leq1$, whenever $t\geq t_0$. From
\eqref{4-76}, it follows that for $t\geq t_0$,
$$
\sigma'(t)=-(I+\tilde \Phi)^{-1}(\sigma(t))\leq -C_0\sigma^{\beta}(t).
$$
Therefore, for all $t \geq t_0$, $\sigma(t)\leq\tilde{\sigma}(t)$, where
$\tilde{\sigma}(t)$ is the solution of
\begin{align}
\begin{cases}
    \tilde{\sigma}'(t)+C_0\tilde{\sigma}^{\beta}(t)=0, \\
    \tilde{\sigma}(t_0)=\sigma(t_0),
  \end{cases}\nonumber
\end{align}
from which we get for all $t \geq t_0$,
$$
\tilde{\sigma}(t)=[C_0(\beta-1)(t-t_0)+\sigma^{1-\beta}(t_0)]^{- \frac{1}{\beta-1}},
$$
which, along with \eqref{4-74}, implies that for any $t\geq
(t_0+1)T$,
$$
\mathcal{E}(t)\leq \sigma\left(\frac{t}{T}-1\right)\leq
\tilde{\sigma}\left(\frac{t}{T}-1\right)=\left[C_0(\beta-1)\left(\frac{t}{T}-1-t_0\right)+\sigma^{1-\beta}(t_0)\right]^{- \frac{1}{\beta-1}}.
$$
Since $\sigma(t_0)$ depends on $\mathcal{E}(0)$, there exists a
positive constant $C(\mathcal E(0))$ depending on $\mathcal{E}(0)$ such that for
any $t\geq0$,
$$
\mathcal{E}(t) \leq C(\mathcal E(0))  (1+t)^{-\frac{1}{\beta-1}}.
$$
This ends the proof.
\end{proof}

\def\cprime{$'$}


\begin{thebibliography}{10}

\bibitem{Avalos2}
G.~Avalos.
\newblock Wellposedness of a structural acoustics model with point control.
\newblock In {\em Differential geometric methods in the control of partial
  differential equations ({B}oulder, {CO}, 1999)}, volume 268 of {\em Contemp.
  Math.}, pages 1--22. Amer. Math. Soc., Providence, RI, 2000.

\bibitem{Avalos1}
G.~Avalos and I.~Lasiecka.
\newblock Uniform decay rates for solutions to a structural acoustics model
  with nonlinear dissipation.
\newblock {\em Appl. Math. Comput. Sci.}, 8(2):287--312, 1998.

\bibitem{Avalos3}
G.~Avalos and I.~Lasiecka.
\newblock Exact controllability of structural acoustic interactions.
\newblock {\em J. Math. Pures Appl. (9)}, 82(8):1047--1073, 2003.

\bibitem{Avalos4}
G.~Avalos and I.~Lasiecka.
\newblock Exact controllability of finite energy states for an acoustic
  wave/plate interaction under the influence of boundary and localized
  controls.
\newblock {\em Adv. Differential Equations}, 10(8):901--930, 2005.

\bibitem{BLR1}
V.~Barbu, I.~Lasiecka, and M.~A. Rammaha.
\newblock On nonlinear wave equations with degenerate damping and source terms.
\newblock {\em Trans. Amer. Math. Soc.}, 357(7):2571--2611 (electronic), 2005.

\bibitem{Beale76}
J.~T. Beale.
\newblock Spectral properties of an acoustic boundary condition.
\newblock {\em Indiana Univ. Math. J.}, 25(9):895--917, 1976.

\bibitem{Becklin-Rammaha1}
A.~R. Becklin and M.~A. Rammaha.
\newblock Global solutions to a structure acoustic interaction model with
  nonlinear sources.
\newblock {\em J. Math. Anal. Appl.}, 487(2):123977, 32 pp, 2020.

\bibitem{Becklin-Rammaha2}
A.~R. Becklin and M.~A. Rammaha.
\newblock Hadamard well-posedness for a structure acoustic model with a
  supercritical source and damping terms.
\newblock {\em Evol. Equ. Control Theory}, 10(4):797--836, 2021.

\bibitem{BL3}
L.~Bociu and I.~Lasiecka.
\newblock Blow-up of weak solutions for the semilinear wave equations with
  nonlinear boundary and interior sources and damping.
\newblock {\em Appl. Math. (Warsaw)}, 35(3):281--304, 2008.

\bibitem{BL2}
L.~Bociu and I.~Lasiecka.
\newblock Uniqueness of weak solutions for the semilinear wave equations with
  supercritical boundary/interior sources and damping.
\newblock {\em Discrete Contin. Dyn. Syst.}, 22(4):835--860, 2008.

\bibitem{BL1}
L.~Bociu and I.~Lasiecka.
\newblock Local {H}adamard well-posedness for nonlinear wave equations with
  supercritical sources and damping.
\newblock {\em J. Differential Equations}, 249(3):654--683, 2010.

\bibitem{Cagnol1}
J.~Cagnol, I.~Lasiecka, C.~Lebiedzik, and J.-P. Zol\'{e}sio.
\newblock Uniform stability in structural acoustic models with flexible curved
  walls.
\newblock {\em J. Differential Equations}, 186(1):88--121, 2002.

\bibitem{Chueshov-1999}
I.~D. Chueshov.
\newblock {\em ``Introduction to the theory of infinite-dimensional dissipative
  systems"}.
\newblock [University Lectures in Contemporary Mathematics]. AKTA, Kharkiv,
  1999.

\bibitem{GT}
V.~Georgiev and G.~Todorova.
\newblock Existence of a solution of the wave equation with nonlinear damping
  and source terms.
\newblock {\em J. Differential Equations}, 109(2):295--308, 1994.

\bibitem{MG1}
M.~Grobbelaar-Van~Dalsen.
\newblock On a structural acoustic model with interface a {R}eissner-{M}indlin
  plate or a {T}imoshenko beam.
\newblock {\em J. Math. Anal. Appl.}, 320(1):121--144, 2006.

\bibitem{MG3}
M.~Grobbelaar-Van~Dalsen.
\newblock On a structural acoustic model which incorporates shear and thermal
  effects in the structural component.
\newblock {\em J. Math. Anal. Appl.}, 341(2):1253--1270, 2008.

\bibitem{Guo98}
Y.~Guo.
\newblock Global well-posedness for nonlinear wave equations with supercritical
  source and damping terms.
\newblock {\em J. Math. Anal. Appl.}, 477(2):1087--1113, 2019.

\bibitem{GR2}
Y.~Guo and M.~A. Rammaha.
\newblock Global existence and decay of energy to systems of wave equations
  with damping and supercritical sources.
\newblock {\em Z. Angew. Math. Phys.}, 64(3):621--658, 2013.

\bibitem{GR}
Y.~Guo and M.~A. Rammaha.
\newblock Systems of nonlinear wave equations with damping and supercritical
  boundary and interior sources.
\newblock {\em Trans. Amer. Math. Soc.}, 366(5):2265--2325, 2014.

\bibitem{GRSTT}
Y.~Guo, M.~A. Rammaha, S.~Sakuntasathien, E.~S. Titi, and D.~Toundykov.
\newblock Hadamard well-posedness for a hyperbolic equation of viscoelasticity
  with supercritical sources and damping.
\newblock {\em J. Differential Equations}, 257(10):3778--3812, 2014.

\bibitem{Howe1998}
M.~S. Howe.
\newblock {\em Acoustics of fluid-structure interactions}.
\newblock Cambridge Monographs on Mechanics. Cambridge University Press,
  Cambridge, 1998.

\bibitem{MOHNICK1}
N.~J. Kass and M.~A. Rammaha.
\newblock Local and global existence of solutions to a strongly damped wave
  equation of the $p$-{L}aplacian type.
\newblock {\em Commun. Pure Appl. Anal.}, 17(4):1449--1478, 2018.

\bibitem{MOHNICK2}
N.~J. Kass and M.~A. Rammaha.
\newblock On wave equations of the {$p$}-{L}aplacian type with supercritical
  nonlinearities.
\newblock {\em Nonlinear Anal.}, 183:70--101, 2019.

\bibitem{LAS1999}
I.~Lasiecka.
\newblock Boundary stabilization of a 3-dimensional structural acoustic model.
\newblock {\em J. Math. Pures Appl. (9)}, 78(2):203--232, 1999.

\bibitem{Las2002}
I.~Lasiecka.
\newblock {\em Mathematical control theory of coupled {PDE}s}, volume~75 of
  {\em CBMS-NSF Regional Conference Series in Applied Mathematics}.
\newblock Society for Industrial and Applied Mathematics (SIAM), Philadelphia,
  PA, 2002.

\bibitem{LTa}
I.~Lasiecka and D.~Tataru.
\newblock Uniform boundary stabilization of semilinear wave equations with
  nonlinear boundary damping.
\newblock {\em Differential Integral Equations}, 6(3):507--533, 1993.

\bibitem{nehari1960class}
Z.~Nehari.
\newblock On a class of nonlinear second-order differential equations.
\newblock {\em Transactions of the American Mathematical Society},
  95(1):101--123, 1960.

\bibitem{nehari1961characteristic}
Z.~Nehari.
\newblock Characteristic values associated with a class of nonlinear
  second-order differential equations.
\newblock {\em Acta Mathematica}, 105(3-4):141--175, 1961.

\bibitem{PS}
L.~E. Payne and D.~H. Sattinger.
\newblock Saddle points and instability of nonlinear hyperbolic equations.
\newblock {\em Israel J. Math.}, 22(3-4):273--303, 1975.

\bibitem{PRT-1}
P.~Pei, M.~A. Rammaha, and D.~Toundykov.
\newblock Local and global well-posedness of semilinear
  {R}eissner--{M}indlin--{T}imoshenko plate equations.
\newblock {\em Nonlinear Anal.}, 105:62--85, 2014.

\bibitem{PRT-p-Laplacain}
P.~Pei, M.~A. Rammaha, and D.~Toundykov.
\newblock Weak solutions and blow-up for wave equations of {$p$}-{L}aplacian
  type with supercritical sources.
\newblock {\em J. Math. Phys.}, 56(8):081503, 30, 2015.

\bibitem{Minimax}
P.~H. Rabinowitz.
\newblock {\em Minimax methods in critical point theory with applications to
  differential equations}, volume~65 of {\em CBMS Regional Conference Series in
  Mathematics}.
\newblock Published for the Conference Board of the Mathematical Sciences,
  Washington, DC; by the American Mathematical Society, Providence, RI, 1986.

\bibitem{Sattiinger68}
D.~H. Sattinger.
\newblock On global solution of nonlinear hyperbolic equations.
\newblock {\em Arch. Rational Mech. Anal.}, 30:148--172, 1968.

\bibitem{Sh}
R.~E. Showalter.
\newblock {\em Monotone operators in {B}anach space and nonlinear partial
  differential equations}, volume~49 of {\em Mathematical Surveys and
  Monographs}.
\newblock American Mathematical Society, Providence, RI, 1997.

\end{thebibliography}
\end{document}